\documentclass{amsart}
\usepackage{amsmath}
\usepackage{amsthm}
\usepackage{amssymb}
\usepackage{pstricks,pst-node}

\theoremstyle{plain}
  \newtheorem{theorem}{\bf Theorem}[section]
  \newtheorem{proposition}[theorem]{\bf Proposition}
  \newtheorem{lemma}[theorem]{\bf Lemma}

\theoremstyle{remark}
  \newtheorem{remark}[theorem]{\bf Remark}

{}

\date{February 2008}
\begin{document}
\title[Weak Dispersive estimates]{Weak Dispersive estimates for Schr\"odinger equations with
long range potentials}

\author{J. A. Barcel\'o, A. Ruiz, L. Vega and M. C. Vilela}
\thanks{
The first and second authors were supported by Spanish Grant
MTM2005-07652-C02-01, and the third and fourth by Spanish Grant
MTM2004-03029.}
\address{J. A. Barcel\'o, ETSI de Caminos, Universidad Polit\'ecnica de
Madrid, 28040, Madrid, Spain.} \email{juanantonio.barcelo@upm.es}
\address{A. Ruiz, Departamento de Matem\'aticas, Universidad Aut\'onoma
de Madrid, 28049, Madrid, Spain.} \email{alberto.ruiz@uam.es}
\address{L. Vega, Departamento de Matem\'aticas,
Universidad del Pa\'{\i}s Vasco, Apartado 644, 48080 Bilbao,
Spain.} \email{luis.vega@ehu.es}
\address{M. C. Vilela, Departamento de Matem\'atica Aplicada,
Universidad de Valladolid, Plaza Santa Eulalia 9 y 11, 40005
Segovia, Spain.} \email{maricruz@dali.eis.uva.es}
 \maketitle{}
\begin{abstract}
We prove some local smoothing estimates for the Schr\"{o}dinger
initial value problem with data in $L^2(\mathbb{R}^d)$, $d \geq
2$ and a general class of potentials. In the  repulsive setting we have to assume just a power like decay $(1+|x|)^{-\gamma}$ for some $\gamma>0$. Also attractive perturbations are considered. The estimates hold for all time and as a consequence a weak dispersion of the solution is obtained. The proofs are based on similar estimates for the
corresponding  stationary Helmholtz equation and Kato H-smooth theory.\end{abstract}
\maketitle
\section{Introduction}
We consider Schr\"odinger Hamiltonians $-\Delta +V(x)$, with $V$ a
real potential, $x \in \mathbb{R}^d$ and $d \geq 2$.  We study some
dispersive estimates, also called local smoothing estimates, for
solutions of the evolution initial value problem
\begin{equation}
\label{sofia1} \left\{
\begin{array}{ll}
i\partial_tu - \Delta_xu +V(x)u=0, \;\;\; (x,t) \in \mathbb{R}^d
\times \mathbb{R} \;\;\;\; d \geq 2,
\\
u(x,0)=u_0(x).
\end{array}
\right.
\end{equation}
under different conditions on the potential $V$.

The first result is concerned with   repulsive potentials. Our
notion of repulsion involves conditions on the sign of $V$ and
$\partial_r V$ (the radial derivative of $V$),  and therefore is more restrictive to the one given for example in   \cite{A} for  which  just conditions on the positive part of  $\partial_r V$ are  assumed.
\begin{theorem}
\label{teorema1} Let $V$ be a real valued function satisfying the
following conditions:
\begin{itemize}
\item[i)] $ V\ge 0, $ \item[ii)] There exists $\gamma > 0$ such
that:
\begin{itemize}
\item If $d>3$, we suppose that there exist $ \eta >0$ such that
\begin{equation}
\label{crepulsivo} \gamma V(x)+ |x| \partial_r V(x) \leq
(1-\eta)\frac{(d-1)(d-3)}{2|x|^2};
\end{equation}
\item For $d=3$, we suppose that there exists  $W(t)\geq 0$ on
$(0,\infty)$ such that
\begin{equation}
\label{crepulsivo1} \sup_{|x|=t} \{ \gamma V(x)+ |x| \partial_r
V(x)\} \leq W(t),
\end{equation}
and
\begin{equation}\label{crepulsivo2}
 \int_0^\infty tW(t)dt <\frac{1}{2}.
\end{equation}
\end{itemize}
\end{itemize}

Let $H=-\Delta+V$, then for $d\ge 3,$ the Schr\"odinger operator
$e^{itH}$ satisfies the following estimates:
\begin{equation}\label{juan10}
\sup_{R>0}\frac{1}{R}\int_{B(0,R)}\int_{- \infty}^{\infty}
|D^{\frac{1}{2}}e^{itH}u_0(x)|^2dt\ dx \leq C\ \|u_0\|_2^2,
\end{equation}
where, for $\alpha \in \mathbb{C}$, the operator $D^{\alpha}$ is
defined by
$\widehat{D^{\alpha}f}(\xi)=|\xi|^{\alpha}\widehat{f}(\xi)$.

We also have
\begin{equation}\label{juan14}
\sup_{R>0}\frac{1}{R}\int_{B(0,R)}V^{\frac{1}{2}}(x)\int_{-\infty}^{\infty}
|e^{itH}u_0(x)|^2dt\ dx \leq C\ \|u_0\|_2^2.
\end{equation}

\end{theorem}

Estimate (\ref{juan10}), usually known as the local smoothing
estimate (see \cite{CS}, \cite{S}, \cite{V} and \cite{RV}),
goes back to the work of Kato and Kruzhkov-Faminski in the context of KdV
equations. Notice that here we obtain an estimate global in time.

 In \cite{BRV} it is proved that  for $d\geq3$ and  potentials $V$ satisfying   the conditions
 (\ref{crepulsivo}) and (\ref{crepulsivo1}) with $\gamma =0$ and (\ref{crepulsivo2})
\begin{equation}
\label{pp}
\begin{array}{ll}
\sup_{R>0}\frac{1}{R}\int_{B(0,R)}\int_{- \infty}^{\infty}
|\nabla_x e^{itH}u_0(x)|^2dt\ dx
\vspace{0.4cm}
\\
+\int_{\mathbb{R}^d} (
\partial_rV)_-(x) \int_{-\infty}^\infty | e^{itH}u_0(x)|^2dt\ dx
 \leq C\ \|u_0\|_{H^{1/2}(\mathbb{R}^d)}^2,
 \end{array}
 \end{equation} where $(\partial_r
V)_-$ is the negative part of $\partial_r V$.
Notice that the estimate requires the initial
 datum to be in the non-homogeneous Sobolev space $H^{\frac{1}{2}}(\mathbb{R}^d)$ which means that from the point of view of small frequencies the result is not so strong. In a sense this can not be avoided because  in \cite{VV} (see also \cite{GMP}) it is proved that under some conditions on the potential the right hand side of (\ref{pp}) can be replaced by
 $$\|u_0\|^2_{\dot{H}_V^{1/2}}: =\|(  { - \Delta
+V})^{1/4}u_0\|^2_{L^2(\mathbb{R}^d)},$$ and that this is optimal.
In our new result (\ref{juan10}) we write the homogeneous derivative  and therefore there is no difference of behaviour between low and high frequencies. We also want to mention that we  answer affirmatively  the question
 posed  in  Remark 4 of \cite{BRV}     for  $L^2$  initial
 data ( see Remark \ref{ultimoremark} to  be  more precise).

 The  approach we follow in this paper  is    based also on  multiplier technics but, as opposite to \cite{BRV}, the estimates  for  the  evolution  equation rely on  estimates  for  the  resolvent  and  the Kato'  smoothing  theory.
In \cite{A} and \cite{BRV} potentials homogeneous of degree zero
are allowed. However Theorem \ref{teorema1} does not apply in this
case.

Some typical examples of potentials satisfying the conditions of
Theorem \ref{teorema1} are, for $V_{\infty}\geq 0$:
\begin{itemize}
\item[i)]  The potential
$V(x)=\frac{V_{\infty}(\frac{x}{|x|})}{|x|^{\gamma}}$ where $
0<\gamma<2 $   and the potential $V(x)=\frac{c}{|x|^{2}}$ with small $c>0$.
 \item[ii)]
$V(x)=\frac{V_{\infty}(\frac{x}{|x|})}{(1+|x|^2)^{\frac{\alpha}{2}}}$
 with $V_{\infty}$
bounded  and  $0 <\alpha$.
\end{itemize}

In the case $V_\infty >c_0>0$ estimate (\ref{juan14}) is
particularly relevant. In fact, assume  for instance
$V(x)=1/|x|^{\alpha}$ with $0<\alpha<2$. As  a consequence (see \cite{rs}, v.IV, page 147)  we have that for $\delta >0$, the following  estimate  for  the  spectral projection operators holds
$$
\sup_{R>0,\; ||f||_{L^2}=1}\frac{1}{R}
\int_{B(0,R)}|P_{(0,\delta]}f(x)|^2\frac{dx}{|x|^{\frac{\alpha}{2}}}
\le C\ \delta,
$$
where $C$ is an absolute constant.
A similar estimate in the  free case  would  mean,
roughly speaking, that  the usual Fourier transform  $\widehat
f(\xi)$ behaves as  $|\xi|^{\beta}$  for $\xi$ close to the origin, where
$$
\beta\ge -\frac{n}{2}+\frac{1}{2}\left(1-\frac{\alpha}{2}\right),
$$
if $0<\alpha<2.$ Since not every   function in $L^2(\mathbb{R}^d)$
  satisfies this condition, the above inequality has to be understood as  a
special feature of the generalized Fourier transform  associated
to the potential $V$.
It would be interesting to know if Strichartz estimate holds in
this case. Recall that the  existence of   $x_0$ such that
$V_\infty(\frac{x_0 }{|x_0|})=0$ implies that  Strichartz  estimates do not hold, see \cite{GVV}
and \cite{PV2}.

Our next result concerns with potentials which are not necessarily
repulsive. We first consider the case where
the radial variation of the attractive part  is not too big. We have the following result.

\begin{theorem}
\label{teorema2} Let $V$ and $n$ be two real valued functions.
Assume that $V$ satisfies the conditions i) and ii) of the Theorem
\ref{teorema1}, and $n$ satisfies:
\begin{itemize}
\item[iii)] $n=n_1+n_2< 0$ where $n_1 \in
L^{\infty}(\mathbb{R}^d)$ and $n_2$  is such that
\begin{equation}
\label{sobolev} \int_{\mathbb{R}^d}|n_2(x)|\ |g(x)|^2\ dx\le c_1
\int_{\mathbb{R}^d}|\nabla g(x)|^2\ dx, \qquad
\end{equation}
for  some $c_1$  with $0<c_1<1$.

\item[iv)] Given $\rho>0$, take $j_0\in\mathbb{Z}$
such that $2^{j_0}<\rho\le 2^{j_0+1}$ , then
\begin{equation}
\label{atractivo} \begin{array}{ll} \beta_{\rho}=\sum_{j\ge
j_0}^{\infty}2^{j+1}\sup_{x\in C_j}\frac{\nabla n(x)\cdot
x}{|x||n(x)|} + \rho\sup_{x\in B(0,\rho)}\frac{\nabla n(x)\cdot
x}{|x||n(x)|}

\vspace{0.4cm}

\\ \textrm{satisfies} \;\; \; \left\{
\begin{array}{ll} \beta_{\rho} <\frac{1}{4} \hspace{2.5cm} \;\;\; d>3

\vspace{0.3cm}

\\\beta_{\rho}+ \int_0^\infty tW(t)dt<\frac{1}{2}  \;\;\;\;\; d=3, \end{array} \right.
\end{array}
\end{equation}
where $C_j=\{x\in\mathbb{R}^d \; : \;\; 2^j < |x| \leq
2^{j+1}\},\;\;\; \forall j\in\mathbb{Z}.$
\end{itemize}
Let $H=-\Delta+V+n,$ then for $d\ge 3$ and $\tau _0 >0$, the
Schr\"odinger operator $e^{itH}$ satisfies the following
estimates:
\begin{equation}
\label{sofia9a} \sup_{R\ge\rho}\frac{1}{R}\int_{B(0,R)}\int_{-
\infty}^{\infty}
|D^{\frac{1}{2}}e^{itH}\mathcal{P}_{\tau_0}u_0(x)|^2 dtdx \leq
C(\beta_{\rho},\tau_0)\ \|u_0\|_2^2,
\end{equation}
\begin{equation}
\label{sofia9b} \tau
\sup_{R\ge\rho}\frac{1}{R}\int_{B(0,R)}\int_{- \infty}^{\infty}
|e^{itH}\mathcal{P}_{\tau}u_0(x)|^2 dtdx \leq
C(\beta_{\rho},\tau_0) \|u_0\|_2^2 \hspace{1cm} \tau \geq \tau_0,
\end{equation}
where $\mathcal{P}_{\tau}$ is the spectral projection operator associated
with   $[\tau,\infty).$

\end{theorem}

We observe that condition (\ref{atractivo}) remains true if $n(x)$
is changed into $\mu n(x)$ for any $\mu >0$. It also applies to
functions homogeneous of degree zero. Typical examples of such
$n(x)$ are as follows:

Take $\omega(x)=\mu \left(1-e^{g(\infty)-g(|x|)} \right)  ,$ where
$\mu>0$ and $g$ is a radial function such that
$$
0\le  g'\le\frac{b}{(1+r)^{\gamma}} \qquad \forall r\ge 0,
$$
with $\gamma>1$ and $b=b(\gamma)>0$ small enough. Under these
conditions, we can apply  Theorem \ref{teorema2} to $n=\omega-\mu.$
As a consequence, we get estimates similar to (\ref{sofia9a}) and
(\ref{sofia9b}) for $e^{itH}\mathcal{P}_{(\mu,\infty)}u_0$ with
$H=-\Delta+V+\omega$.

In the above example $\omega\leq 0$, therefore for $\mu$  big enough non trivial eigenvalues can be expected   and in that case estimates (9)  and  (10)  can  not
be   extended  to  all $\tau$. Notice however than one would expect the estimate on
 the spectral projection to be true for $\mathcal{P}_{(\tau,\infty)}$ with $\tau>0$ and not
  just for $\mathcal{P}_{(\mu,\infty)}$. This motivates our final result where we   use a
compactness argument that  follows from a uniqueness
theorem due to Ikebe and Saito \cite{IS}.  Some  extra  assumptions are necessary, and in particular the result does not apply to potentials homogeneous of degree zero.  An extra bonus is that we can
consider  the case $d \geq 2$. We prove the following theorem.

\begin{theorem}
\label{teorema3} Let $V_1$ and $V_2$ be two real valued functions
satisfying that there exist  two constants $a>0$ and
$\gamma>0$ such that
\begin{eqnarray}
\label{c1s} |V_1|(x) &\le&
 \frac{a}{(1+|x|)^{\gamma}},
\\
\label{c2s} |V_2|(x) &\le&
 \frac{a}{(1+|x|)^{\gamma+1}},
\\
\label{c3s}
\partial_r V_1(x)
&\le&
 \frac{a}{(1+|x|)^{\gamma+1}},
\end{eqnarray}
and $H=-\Delta+V_1+V_2.$

Then for $d\ge 2$, $\tau_0 >0,$ and $\alpha>0$
  the following
estimates hold:
\begin{equation}
\label{sofia12a} \int_{- \infty}^{\infty}
\|D^{\frac{1}{2}}e^{itH}\mathcal{P}_{\tau_0}u_0(x)\|^2_{L^2((1+|x|)^{-1-\alpha})}dt
\leq C_1(\tau_0)\ \|u_0\|_2^2,
\end{equation}
\begin{equation}
\label{sofia12b} \tau \int_{- \infty}^\infty
\|e^{itH}\mathcal{P}_{\tau }u_0(x)\|^2_{L^2((1+|x|)^{-1-\alpha})}dt
\leq C_2(\tau_0) \|u_0\|_2^2, \hspace{1cm} \tau \geq \tau_0.
\end{equation}

\end{theorem}

Consider as before $\omega(x)=\mu(e^{g(|x|)-g(\infty)}-1),$ where
 $\mu>0$ and $g$ is a radial function such that
 $$0\le g'\le\frac{c}{(1+r)^{\beta}} \qquad \forall r\ge 0.$$
Then if $\beta>1$ , $c>0$ and  we put $V_1=\omega$  we can apply
 Theorem \ref{teorema3}. As a
 consequence, we get estimates similar to (\ref{sofia12a}) and
 (\ref{sofia12b}) for $e^{itH}\mathcal{P}_{(\tau,\infty)}u_0$ with
 $H=-\Delta+\omega$ and $\tau>0$.

Notice that we can consider repulsive perturbations which are long range. In
particular, we extend and improve some of the results  in \cite{ISch}.
On the other hand we have to exclude the case $\tau_0=0$. This
is obvious because we can consider potentials with decay as
$\frac{C}{|x|^2}$, and therefore there can exist 0-eigenfunctions. Take for example  $u(x)=(1+|x|^2)^{\lambda}$  with $\lambda<- d/2$  and  $ V(x)=\frac{\Delta u}{u}$.

As it  was pointed out before, the proofs of the above  theorems are based on similar estimates
for the corresponding stationary Helmholtz equations and  the general Kato's
argument on smoothing operators  as  in \cite{ky} and \cite{RS}. Actually  we prove the so called   supersmoothing estimates which are  stronger, see \cite{ky}.    The procedure  we use   to obtain
  these  stationary estimates  are basically the multiplier
method, see \cite{LS}, \cite{M} and \cite{R}. Similar  estimates  were obtained  in  \cite{PV1}.

We  prove Theorem \ref{teorema1}  in  section  2,
 Theorem \ref{teorema2} in  section 3 and Theorem
  \ref{teorema3} in  section 4.  For  completeness we
   include     Appendix 1 with some useful identities   and
   Appendix 2  with some calculations that  will be needed in the proof of Theorem \ref{teorema1}.

\vskip 1cm \noindent {\bf Notation.}


  We shall make use of the spaces introduced in \cite{AH} which are given by the following norms
$$\|u\|^2_{X}= \sup_{R>0}\frac{1}{R} \int_{B(0,R)} |u(x)|^2dx.$$
We  shall replace the  norm in  the corresponding  predual  space for  the equivalent expression:
$$\|f\|_{X^*}=\sum_{j \in\mathbb{Z}}\left(2^{j+1} \int_{C_j}|f(x)|^2dx\right)^{\frac{1}{2}}.$$
Notice that
$$
\int_{\mathbb{R}^d}f(x)g(x)dx \le C\ \|f\|_{X}\|g\|_{X^*}.
$$ where $C$ is a positive constant depending only on  $d$.

For any $\rho>0$ such that $2^{j_0}<\rho\le 2^{j_0+1},$
we write
$$\|u\|^2_{X_{\rho}}= \sup_{R\ge\rho}\frac{1}{R} \int_{B(0,R)}
|u(x)|^2dx,$$ and
$$
\|f\|_{X_{\rho}^*}=\sum_{j\ge j_0}\left(2^{j+1}
\int_{C_j}|f(x)|^2dx\right)^{\frac{1}{2}}
+\left(\rho\int_{B(0,\rho)}|f(x)|^2dx\right)^{\frac{1}{2}}.
$$
In this case, we also have
$$
\int_{\mathbb{R}^d}u(x)f(x)dx \le C\
\|u\|_{X_{\rho}}\|f\|_{X_{\rho}^*}.
$$

\section{Repulsive potentials}
We start  by proving  a priori  estimates for  the  resolvent of  the operator $H$ and  then
we proceed to prove  Theorem \ref{teorema1}

\subsection{Estimates for the resolvent}

\begin{theorem}
\label{trepulsivo} Let $V$ be a real valued function satisfying the
conditions of  Theorem \ref{teorema1}, and let $u$  be  a
solution  of  the  equation
\begin{equation}
\label{ecuacionrepulsiva} -\Delta u+V(x)\ u\pm i\epsilon u - \tau
u=f, \quad \epsilon\neq 0,\ x\in\mathbb{R}^d\ (d\ge 3).
\end{equation}
Then, for any $\tau\in \mathbb{R},$   the following a
priori estimate holds
\begin{equation}\label{basic}
\begin{array}{ll}  \|\nabla u\|^2_{X} +  \max\{0,\tau\} \|u\|^2_{X} +
(d-3)\int_{\mathbb{R}^d}\frac{|u(x)|^2}{|x|^3}dx +
\int_{\mathbb{R}^d}\frac{V(x)}{|x|}|u(x)|^2 dx

\vspace{0.4cm}

\\
  +\sup_{R>0}\frac{1}{R^3}\int_{B(0,R)}|u(x)|^2dx
 +
\chi(d)\int_{\mathbb{R}^3}\frac{W(|x|)}{|x|}|u(x)|^2dx \leq C\
\|f\|_{X^*}^2,
 \end{array}
\end{equation}

where $C$ is a constant independent of $\epsilon$ and $\tau$, $W$
is the function of Theorem \ref{teorema1} and $\chi(d)$ is  defined by
\begin{equation}\label{tonto7}
\hspace{3cm} \chi(d)= \left\{ \begin{array}{ll} 1 \hspace{1.2cm}
\text{if} \;\;\; d=3, \\ 0 \hspace{1.2cm} \text{if} \;\;\; d\neq
3.
\end{array} \right.
\end{equation}
\end{theorem}

\begin{remark}
From estimate (\ref{basic}) follows
\begin{equation}\label{agosto1}
\| \nabla R_H(\tau \pm i \epsilon) f\|_{X} \leq C \|f\|_{X^*}.
\end{equation}
 and
\begin{equation}\label{agosto2}
\|R_H(\tau\pm i\epsilon)f\|_{L^2\left(\frac{V(x)}{|x|}\right)}
\leq  C\ \|f\|_{X^*},
\end{equation}
where $C$ is  independent of $\epsilon \neq 0$ and $\tau$ and $R_H$ denotes
the resolvent operator  of $H=- \Delta + V$.
\end{remark}

\proof

 By a density argument we might assume  the  a priori condition $f\in L^2(\mathbb{R}^d).$

The proof is based on the following estimates for the solution of
(\ref{ecuacionrepulsiva}):

\begin{equation}\label{maricruz0}
\left\{ \begin{array}{lll} \|\nabla u\|_X^2 +\int_{\mathbb{R}^d}
\frac{V(x)}{|x|}|u(x)|^2\ dx +\sup_{R>0}\frac{1}{R^3}\int_{B(0,R)}
|u(x)|^2\ dx

\vspace{0.3cm}

\\+
(d-3) \int_{\mathbb{R}^d} \frac{|u(x)|^2}{|x|^3}dx + \chi(d)
\int_{\mathbb{R}^3}\frac{W(|x|)}{|x|}|u(x)|^2dx

\vspace{0.4cm}

\\ \leq C \|f\|^2_{X^*}+\frac{1}{4}(|\epsilon | + \max \{ 0,
\tau\})\|u\|^2_X, \end{array} \right.
\end{equation}

\begin{equation}
\label{tau<epsilon} |\epsilon| \|u\|^2_X \leq  \|f\|^2_{X^*} +
\sup_{R>0}\frac{1}{R^3}\int_{B(0,R)}|u(x)|^2dx +  \|\nabla
u\|^2_X.
\end{equation}
\begin{equation}\label{principal2repulsivo}
\left\{ \begin{array}{lll}  \|\nabla u\|^2_{X} + \tau\ \|u\|^2_{X}
 +
\int_{\mathbb{R}^d}\frac{V(x)}{|x|}|u(x)|^2\ dx
 +
 \sup_{R>0}\frac{1}{R^3}\int_{B(0,R)}|u(x)|^2dx

\vspace{0.3cm}

\\+
 \|V^{\frac{1}{2}}u\|^2_{X}+ (d-3) \int_{\mathbb{R}^d}
\frac{|u(x)|^2}{|x|^3}dx +
\chi(d)\int_{\mathbb{R}^3}\frac{W(|x|)}{|x|}|u(x)|^2dx

\vspace{0.4cm}

\\ \leq C \|f\|^2_{X^*} +\frac{1}{2}|\epsilon|\ \|u\|^2_{X} ,
\hspace{0.5cm} \tau \geq 0.
\end{array} \right.
\end{equation}

Estimate (\ref{basic}) follows  easily from  the  three above
estimates,  assuming  that  the  terms in the  left hand side are
finite (this will be  seen through the proofs of these estimates)
in  fact: for $\tau < 0$, (\ref{basic}) is a consequence of
(\ref{maricruz0}) and (\ref{tau<epsilon}) and for $\tau \geq 0$,
(\ref{basic}) is a consequence of (\ref{principal2repulsivo}) and
(\ref{tau<epsilon}).


\noindent \underline{Proof of (\ref{maricruz0})}.
The  basic  estimates  to  obtain (\ref{maricruz0}) are (\ref{alberto3})  and  (\ref{alberto1}) in  the  appendix.  The point  is  the  choice  of  an appropriate   radial function $\Phi_R (x) \equiv \Phi_R(|x|)$ in (\ref{alberto3}) so  that   its  left hand  side to be   an upper  bound of the  left hand  side  of  (\ref{maricruz0}).

We  start  by  performing some  calculations  of  derivatives  of
such a radial  function. We have:
\begin{equation}\label{anselmo3}
\nabla \bar{u}(x) \cdot D^2 \Phi_R(x) \cdot \nabla u(x)=
\Phi_R''(|x|)
|\partial_ru(x)|^2+\frac{\Phi_R'(|x|)}{|x|}|\partial_\tau u(x)|^2,
\end{equation}
where  $\partial_ru$    and $\partial_\tau u$ are respectively the radial and
tangential parts of the derivative of $u$.

We also have
\begin{equation}\label{granveg} - \nabla \Phi_R (x) \cdot \nabla V(x) =-\Phi_R'(|x|) \partial_rV(x).
\end{equation}
 Roughly  speaking, we  want the quadratic form determined by $D^2 \Phi_R(x)$ to be  positive   and the function
  $ \frac{1}{4} \Delta^2\Phi_R(x) +
\frac{1}{2}\nabla V(x) \cdot \nabla\Phi_R (x)$ to be nonpositive.

In  the  case of   $d>3$ we proceed to  find  $\Phi_R$. We will
use a combination of the classical Morawetz multiplier $|x|$, see
\cite{M}, and the multiplier $(1+|x|^2)^{1/2}$ ( adapted to  the
Hemholtz equation). More precisely,  we start by considering
$$\Phi(x)=(1+|x|^2)^{1/2}+|x|.$$
As  above  we identify $\Phi(x) \equiv \Phi(r)$ with $|x|=r$. Then  we have
\begin{equation}\label{maricruz1}
 \Delta^2 \Phi (x) =-\frac{(d-1)(d-3)}{(1+|x|^2)^{3/2}}
 -\frac{6(d-3)}{(1+|x|^2)^{5/2}}
-\frac{15}{(1+|x|^2)^{7/2}}-\frac{(d-1)(d-3)}{|x|^3},
\end{equation}
\begin{equation}\label{paula1}
-\Delta^2\Phi(x) \ge \frac{d(d+2)}{8\sqrt 2 }\chi_{B(0,1)}(x) +
\frac{(d-1)(d-3)}{|x|^3},
\end{equation}
\begin{equation}\label{maricruz2}
1 \leq \Phi'(r)=\frac{r}{(1+r^2)^{1/2}}+1 \leq 2  \hspace{1cm} r
\geq 0,
\end{equation}
\begin{equation}
\label{maricruz3} 0 \leq
r\Phi''(r)=\frac{r}{(1+r^2)^{1/2}}-\frac{r^3}{(1+r^2)^{3/2}} \leq
1 \hspace{1cm} r \geq 0,
\end{equation}
\begin{equation}
\label{maricruz4} \frac{1}{2\sqrt{2}} \leq \inf \left\{\Phi''(r),
\frac{\Phi'(r)-1}{r} \right\} \hspace{1cm} r \leq 1 .
\end{equation}

Now we define for $R>0$
$$\Phi_R(x)=R\Phi(\frac{x}{R}).$$
As a consequence of (\ref{crepulsivo}) and (\ref{paula1}):
\begin{equation}\label{villalobos2}
\begin{array}{ll}  \hspace{0.75cm} - \frac{1}{4} \Delta^2 \Phi_R (x) -\frac{1}{2} \nabla \Phi_R
(x) \cdot \nabla V (x)

\vspace{0.3cm}

\\
\geq \frac{\gamma}{2} \frac{V(x)}{|x|}+
\frac{d(d+2)}{32\sqrt{2}R^3}\chi_{B(0,R)}(x)+\frac{\eta}{4}
\frac{(d-1)(d-3)}{|x|^3}. \end{array}
\end{equation}

 By using
(\ref{paula1})-(\ref{maricruz4}) and (\ref{anselmo3}), we obtain, by inserting all
 the  estimates in (\ref{alberto3}),
that
$$ \frac {1}{ R}\int_{B(0,R)}|\nabla u(x)|^2\ dx
+\int_{\mathbb{R}^d}\frac {V(x)}{|x|}|u(x)|^2\ dx$$
\begin{equation}\label{control1}
+(d-3)\int_{\mathbb{R}^d}\frac{|u(x)|^2}{|x|^3}dx
+\frac{1}{R^3}\int_{B(0,R)} |u(x)|^2\ dx
\end{equation}
$$
\leq
 C\int_{\mathbb{R}^d} |f(x)||\nabla u(x)|\ dx
+C\int_{\mathbb{R}^d}|f(x)|\frac{|u(x)|}{|x|}\ dx
 +C| \epsilon | \int_{\mathbb{R}^d}|u(x)|\
|\nabla u(x)|\ dx.$$

 To  finish  we  need    the terms on the left hand side of
(\ref{control1}) to be finite but,  before seeing  this, let us
deal with    the case $d=3$. In this case, we can find, for $R>0$,
a radial function $\Phi$ (see Lemma 6.1 in Appendix 2) such that:

\begin{itemize}
\item $\Delta^2
\Phi(x)=-\frac{c_1}{R^3}\chi_{(0,R)}(x)-\frac{W(|x|)}{|x|}, \;\; x
\in \mathbb{R}^3$, where $W$ is the  function in
(\ref{crepulsivo1}). \item $\inf_{r>0} \{ \Phi'(r),\Phi''(r) \}
\geq 0$. \item $\inf_{r \in (0,R)} \left\{
\frac{\Phi'(r)}{r},\Phi''(r) \right\} \geq \frac{c_2}{R}$. \item
$c_3<\Phi'(r)< \kappa < \frac{1}{2}, \;\;\; r>0$.
\end{itemize}
for some $c_1$, $c_2$, $c_3$ and $\kappa$ positive constants.

If we use the above inequalities and (\ref{crepulsivo1}), we have
\begin{equation}\label{dionisio1}
\begin{array}{lll} -\frac{1}{4} \Delta^2 \Phi(x)-\frac{1}{2} \nabla
\Phi(x) \cdot V(x)=
\frac{c_1}{4R^3}\chi_{(0,R)}(x)+\frac{W(|x|)}{4|x|}-\frac{1}{2}\Phi'(x)
\partial_r  V(x)

\vspace{0.3cm}

\\ \hspace{1.5cm} \geq \frac{c_1}{4R^3}\chi_{(0,R)}(x) +
\frac{ \gamma V(x)}{2|x|}+ \frac{w(|x|)}{2|x|} \left(
\frac{1}{2}-\Phi'(|x|) \right)

\vspace{0.4cm}

\\ \hspace{2cm} \geq \frac{c_1}{4R^3}\chi_{(0,R)}(x) +  \frac{ \gamma
V(x)}{2|x|} + \frac{w(|x|)}{2|x|} \left( \frac{1}{2}-\kappa
\right) .\end{array}
\end{equation}

Now, since $\frac{1}{2}-\kappa >0$, we can  use (\ref{alberto3})
 as in the case $d>3$ to conclude   that for   $d \geq 3$ we have

$$ \frac {1}{ R}\int_{B(0,R)}|\nabla u(x)|^2\ dx
+\int_{\mathbb{R}^d}\frac {V(x)}{|x|}|u(x)|^2\ dx
+(d-3)\int_{\mathbb{R}^d}\frac{|u(x)|^2}{|x|^3}dx$$
\begin{equation}\label{control1b}
+ \chi(d) \int_{\mathbb{R}^3}\frac{W(|x|)}{|x|}|u(x)|^2dx
+\frac{1}{R^3}\int_{B(0,R)} |u(x)|^2\ dx
\end{equation}
$$
\leq
 C\int_{\mathbb{R}^d} |f(x)||\nabla u(x)|\ dx
+C\int_{\mathbb{R}^d}|f(x)|\frac{|u(x)|}{|x|}\ dx
 +C| \epsilon | \int_{\mathbb{R}^d}|u(x)|\
|\nabla u(x)|\ dx.$$

We check now that the terms on the left hand side of
(\ref{maricruz0}) are finite.  Notice  that, since the right hand
side of (\ref{control1b}) does not depend on $R>0$,  it suffices
to  check that it is finite.

 On one hand the classical theory guarantees, with our  a priori  condition $f\in L^2$,   the existence and uniqueness of
solution of (\ref{ecuacionrepulsiva}) in $L^2(\mathbb{R}^d).$
On the other hand, by  taking $\varphi =1$ in (\ref{alberto1}),
since $V\ge 0,$ we obtain
\begin{equation}
\label{alberto5} \int |\nabla u|^2 dx \leq (\max \{0,\tau \}+1)
\|u\|_2^2 + \|f\|_2^2,
\end{equation}
and hence  we obtain   that $u\in W^{1,2}.$
Therefore,
\begin{equation}
\label{julio0} \int_{\mathbb{R}^n} |f(x)||\nabla u(x)|\ dx \leq
\|f\|_{L^2}\|\nabla u\|_{L^2}<\infty,
\end{equation}
\begin{equation} \label{julio00} \int_{\mathbb{R}^n} |u(x)||\nabla
u(x)|\ dx \leq \|u\|_{L^2}\|\nabla u\|_{L^2}<\infty,
\end{equation}

and using  Hardy's inequality, we have
\begin{equation}
\label{julio1}
 \int_{\mathbb{R}^n}|f(x)|\frac{|u(x)|}{|x|}\ dx
  \leq C   \int_{\mathbb{R}^d}|f(x)|^2dx +
C\int_{\mathbb{R}^d}|\nabla u(x)|^2dx < \infty.
\end{equation}
 Finally, by taking sup  in
 (\ref{control1b}) we get
\begin{equation}\label{alejandra5}
\begin{array}{ll}
  \||\nabla
u \|_X^2 +\int_{\mathbb{R}^d}\frac {V(x)}{|x|}|u(x)|^2\ dx +
\chi(d) \int_{\mathbb{R}^3}\frac{W(|x|)}{|x|}|u(x)|^2dx

\vspace{0.4cm}

\\ + \sup_{R>0} \frac{1}{R^3}\int_{B(0,R)} |u(x)|^2\ dx +
(d-3)\int_{\mathbb{R}^d} \frac{|u(x)|^2}{|x|^3}\ dx < \infty
\;\;\;\;\; d \geq 3,
\end{array}
\end{equation}
 as  desired.
 Now, to obtain  the  a priori  estimate, we
bound   the terms on the right hand side of (\ref{control1b}).

We start by  writing
\begin{equation}\label{alejandra2}
\begin{array}{llll}  C \int_{\mathbb{R}^d}|f(x)|\frac{|u(x)|}{|x|}\ dx \leq
C\sum_{j \in \mathbb{Z}}\int_{C(j)}|f(x)|\frac{|u(x)|}{|x|}dx

\vspace{0.4cm}

\\
 \leq C\sum_{j \in \mathbb{Z}} \left(2^j \int_{C(j)}|f(x)|^2dx
\right)^{1/2}\left(\frac{1}{2^{3j}}\int_{C(j)}|u(x)|^2 dx
\right)^{1/2}

\vspace{0.4cm}

\\
 \leq C \|f\|_{X^*}\left( \sup_{R>0}\frac{1}{R^3}\int_{B(0,R)}
|u(x)|^2\ dx \right)^{1/2}

\vspace{0.4cm}

\\  \leq C\|f\|^2_{X^*} +\frac{1}{2}
\sup_{R>0}\frac{1}{R^3}\int_{B(0,R)} |u(x)|^2\ dx. \end{array}
\end{equation}
 We also have the pairing
\begin{equation}\label{alejandra2b}
C \int_{\mathbb{R}^d} |f(x)||\nabla u(x)|\ dx \leq C
\|f\|_{X^*}\|\nabla u\|_X \leq C\|f\|^2_{X^*} +\frac{1}{2}\|\nabla
u\|^2_X.
\end{equation}

To treat  the third term on the right hand side of
(\ref{control1b}), we take
  $\varphi =1$ in (\ref{alberto2})    to obtain
\begin{eqnarray}
\label{fi1imaginaria} |\epsilon| \int_{\mathbb{R}^d} |u(x)|^2 dx
&\leq& \int_{\mathbb{R}^d} |f(x)| |u(x)|dx.
\end{eqnarray}
and if we again take $\varphi =1$ in (\ref{alberto1}) and by using
(\ref{fi1imaginaria}), we   get
\begin{equation}\label{fi1real}
\begin{array}{ll}
|\epsilon | \int_{\mathbb{R}^d}|\nabla u(x)|^2\ dx \leq \max\{ 0,
\tau \} | \epsilon | \int_{\mathbb{R}^d} |u(x)|^2 dx +| \epsilon|
\int_{\mathbb{R}^d} |f(x)| |u(x)|\ dx

\vspace{0.4cm}

\\ \hspace{2cm}
\leq  (\max\{ 0, \tau \}+| \epsilon |)\int_{\mathbb{R}^d} |f(x)|
|u(x)|\ dx. \end{array}
\end{equation}

Now we use  Cauchy-Schwartz inequality, (\ref{fi1imaginaria}),
(\ref{fi1real})
 to have
\begin{equation}\label{julio2}
C |\epsilon| \int_{\mathbb{R}^d}|u(x)||\nabla u(x)|dx \leq
C\|f\|^2_{X^*}+ \frac{1}{4} ( |\epsilon| + \max\{0,\tau\})
\|u\|^2_X.
\end{equation}

 From this, (\ref{control1b}),  (\ref{alejandra2}) and by taking
sup in  $R>0$, we have
\begin{equation}\label{alejandra3}
\begin{array}{lll}
\| \nabla u\|^2_X  +\int_{\mathbb{R}^d}\frac {V(x)}{|x|}|u(x)|^2\
dx +  \sup_{R>0} \frac{1}{R^3}\int_{B(0,R)} |u(x)|^2\ dx

\vspace{0.4cm}

\\ +(d-3) \int_{\mathbb{R}^d}\frac{|u(x)|^2}{|x|^3}dx+
\chi(d) \int_{\mathbb{R}^3}\frac{W(|x|)}{|x|}|u(x)|^2dx \leq
C\|f\|_{X^*}

\vspace{0.4cm}

\\ +  \frac{1}{2} \sup_{R>0}\frac{1}{R^3}\int_{B(0,R)}
|u(x)|^2\ dx +\frac{1}{2}\|\nabla u\|^2_X+
 \frac{1}{4} ( |\epsilon| +\max\{0,\tau\}) \|u\|^2_X. \end{array}
\end{equation}
and we obtain (\ref{maricruz0}).

\noindent \underline{Proof of of estimate (\ref{tau<epsilon})}.

For $R>0$ fixed, we consider the function
$$
\varphi_R(x) = \frac{1}{R}\chi_{\{|x|<R\}}(x) +
\left(\frac{2}{R}-\frac{|x|}{R^2}\right)\chi_{\{R \leq |x| <
2R\}}.
$$
We have that
$$
\frac{1}{R}\chi_{\{|x|<R\}}(x) \le \varphi_R(x) \le
\frac{1}{R}\chi_{\{|x|<2R\}}(x),
$$
and
$$
|\nabla\varphi_R(x)|\le \frac{1}{R^2}\chi_{\{R<|x|<2R\}}(x).
$$

If we take in (\ref{alberto2}) $\varphi=\varphi_R,$ using the
previous estimates we obtain that
\begin{equation}\label{luis0}
\begin{array}{ll}
 \frac {|\epsilon|}{ R} \int_{B(0,R)}|u(x)|^2\ dx \leq
\frac 1 {R^2}\int_{R<|x|<2R}|\nabla u(x)||u(x)|\ dx

\vspace{0.4cm}

\\ \hspace{1.8cm}
+ \frac 1 R \int_{B(0,2R)}|f(x)||u(x)|\ dx. \end{array}
\end{equation}

Using  Cauchy-Schwartz inequality,
 we
have
\begin{equation}
\label{first} \frac 1 {R^2}\int _{R<|x|<2R}|\nabla u(x)||u(x)|dx
\le \frac{1}{4} \|\nabla u\|_X^2 + \frac{1}{4}
\sup_{R>0}\frac{1}{R^3} \int_{B(0,R)}|u(x)|^2dx.
\end{equation}

Finally, as in (\ref{alejandra2})
\begin{equation}\label{second}
\begin{array}{ll} \frac{1}{R}\int_{B(0,2R)}|f(x)||u(x)|dx \leq
2\int_{\mathbb{R}^d}|f(x)|\frac{ |u(x)|}{|x|}dx

\vspace{0.4cm}

\\ \hspace{0.5cm}
\leq \frac{1}{2}  \|f\|_{X^*}^2 + \frac{1}{2}
\sup_{R>0}\frac{1}{R^3} \int_{B(0,R)}|u(x)|^2dx. \end{array}
\end{equation}

The result follows from (\ref{luis0}), (\ref{first}) and
(\ref{second}). \hfill $\square$

\noindent \underline{Proof of of estimate
(\ref{principal2repulsivo})}.

 Following \cite{PV1}, for $R>0$
fixed, we consider two functions $\Phi_R$ and $\varphi_R$ given by
\begin{equation}
\label{PhiR} \nabla \Phi_R(x) =\frac{x}{R}\chi_{\{|x|<
R\}}(x)+\frac {x}{|x|}\chi_{\{|x|\ge R\}}(x),
\end{equation}
\begin{equation}
\label{varphiR} \varphi_R(x) =\frac{1}{2R}\chi_{\{|x|< R\}}(x).
\end{equation}
Some calculations give us the following identities which hold in
the distributional sense:
$$
\Delta\Phi_R(x) = \frac{d}{R}\chi_{\{|x|< R\}}(x) +
\frac{d-1}{|x|} \chi_{\{|x|> R\}}(x),
$$
$$
D^2_{ij}\Phi_R(x) = \frac{\delta_{ij}}{R}\chi_{\{|x|< R\}}(x) +
\left(\frac{\delta_{ij}}{|x|}-\frac{x_ix_j}{|x|^3}\right)\chi_{\{|x|>
R\}}(x),
$$
and
$$
\Delta( 2\varphi_R -\Delta \Phi_R)(x) = \frac {(d-1)}{R^2}
\delta_{\{|x|=R\}} + \frac{(d-1)(d-3)}{|x|^3} \chi_{\{|x|>
R\}}(x),
$$
where $$\delta_{ij}=\left\{
\begin{array}{ll} 1 \;\;\;\; \textrm{if} \;\;\; i=j \\ 0 \;\;\;\; \textrm{if} \;\;\; i \neq
j. \end{array} \right. $$

It is easy to check  that
$$
\nabla \bar{u}(x) \cdot D^2 \Phi_R(x)\cdot\nabla  u(x) \ge
\frac{|\nabla u(x)|^2}{R}\chi_{\{|x|<R\}}(x),
$$
$$
|\varphi_R|\le \frac{1}{2|x|}, \quad |\nabla\Phi_R|\le 1, \quad
|\Delta\Phi_R|\le \frac{d}{|x|}.
$$

From (\ref{alejandra5})  we know that
$$(d-3)\int_{\mathbb{R}^d}\frac{|u(x)|^2}{|x|^3}dx + \chi(d)
  \int_{\mathbb{R}^3}\frac{W(|x|)}{|x|}|u(x)|^2dx < \infty.$$
 Then using
(\ref{crepulsivo}) and (\ref{crepulsivo1}),
\begin{equation}\label{guillermo1}
\begin{array}{ll}
-\frac{1}{2}\int_{\mathbb{R}^d} \nabla V(x) \cdot \nabla \Phi_R(x)
|u(x)|^2dx  \geq \frac{\gamma}{2R}\int_{B(0,R)}V(x)|u(x)|^2dx

\vspace{0.4cm}

\\  +
\frac{\gamma}{2}\int_{|x|>R}\frac{V(x)}{|x|}|u(x)|^2dx -(1-
\eta)\frac{(d-1)(d-3)}{4}\int_{\mathbb{R}^d}\frac{|u(x)|^2}{|x|^3}dx
\;\;\;\; d >3
\end{array}
\end{equation}
and
\begin{equation}\label{guillermo2}
\begin{array}{ll}
-\frac{1}{2}\int_{\mathbb{R}^3} \nabla V(x) \cdot \nabla \Phi_R(x)
|u(x)|^2dx  \geq \frac{\gamma}{2R}\int_{B(0,R)}V(x)|u(x)|^2dx

\vspace{0.4cm}

\\  +
\frac{\gamma}{2}\int_{|x|>R}\frac{V(x)}{|x|}|u(x)|^2dx
-\frac{1}{2}\int_{\mathbb{R}^3}\frac{W(|x|)}{|x|}|u(x)|^2dx.
\end{array}
\end{equation}
If $d \geq 3$, with the above  inequalities, by using
$\Phi=\Phi_R$ and $\varphi=\varphi_R$ in (\ref{alberto4}) and that
$\tau \geq 0$, we have that
\begin{equation}\label{acotacionstep2}
\begin{array}{llll}
\frac 1{2R}\int_{B(0,R)}|\nabla u(x)|^2 dx +
\frac{(d-1)}{4R^2}\int_ {|x|=R}|u(x)|^2d\sigma_R(x) +
\frac{\tau}{2R}  \int_{B(0,R)}  |u(x)|^2 dx

\vspace{0.4cm}

\\
+ \frac{d-3}{4}\int_{|x|>R}\frac{|u(x)|^2}{|x|^3}\ dx + \frac
{\gamma}{2R}\int_{ B(0,R)}V|u(x)|^2 dx+ \frac
{\gamma}{2}\int_{|x|>R}\frac{V(x)}{|x|}|u(x)|^2\ dx

\vspace{0.4cm}

\\
-(1- \eta)\frac{d-3}{4}\int_{\mathbb{R}^d}\frac{|u(x)|^2}{|x|^3}dx
- \frac{1}{2} \chi(d)
\int_{\mathbb{R}^3}\frac{W(|x|)}{|x|}|u(x)|^2dx \leq
\int_{\mathbb{R}^d}  |f||\nabla u(x)|\ dx

\vspace{0.4cm}

\\
+ \frac{(d+1)}{2}\int_{\mathbb{R}^d}  |f(x)|\frac{|u(x)|}{|x|}\ dx
 + |\epsilon| \int_{\mathbb{R}^d} |\nabla
u(x)||u(x)|\ dx + \frac
{1}{2}\int_{\mathbb{R}^d}\frac{V(x)}{|x|}|u(x)|^2\ dx \end{array}
\end{equation}

If we use (\ref{alberto5})-(\ref{julio1}) and (\ref{alejandra5}),
we have
\begin{equation}\label{guillermo3}
\begin{array}{ll}
\|\nabla u\|^2_X+\int_ {\mathbb{R}^d} \frac{V(x)}{|x|}\
|u(x)|^2dx+\sup_{R>0}\frac{1}{R^2}\int_
{|x|=R}|u(x)|^2d\sigma_R(x)

\vspace{0.4cm}

\\ \hspace{3cm}  + \|V^{\frac{1}{2}}u\|^2_X +\tau \|u\|^2_X < \infty.
\end{array}
\end{equation}
Let $\delta >0$ be such that $$\int_0^\infty t W(t) dt <
\frac{1}{2}-\delta,$$ use
 (\ref{alejandra2}),
(\ref{alejandra2b}), (\ref{julio2}), (\ref{maricruz0}) and take
sup  in (\ref{acotacionstep2}), then  we have
\begin{equation}\label{julio3}
\begin{array}{lll}
 \|\nabla u \|^2_{X} + \sup_{R>0}  \frac{1}{R^2}
\int_ {|x|=R}|u(x)|^2d\sigma_R(x) + \tau \|u\|^2_X+
\|V^{\frac{1}{2}}u\|^2_X

\vspace{0.4cm}

 \\    +
\int_{\mathbb{R}^d}\frac{V(x)}{|x|}|u(x)|^2\ dx + (d-3)
\int_{\mathbb{R}^d} \frac{|u(x)|^2}{|x|^3}dx-\frac{1}{2} \chi(d)
\int_{\mathbb{R}^3}\frac{W(|x|)}{|x|}|u(x)|^2dx

\vspace{0.4cm}

\\
 \leq C \|f\|_{X^*}^2 +\frac{1}{2} \sup_{R>0} \frac{1}{R^3}
\int_{B(0,R)}|u(x)|^2dx  +\frac{1}{2} \|\nabla u \|_{X}^2 +
\frac{1}{2} (|\epsilon|+  \tau) \| u \|_{X}^2. \end{array}
\end{equation}
Estimate (\ref{principal2repulsivo}) follows from this and by
using
\begin{equation}\label{guillermo5}
\sup_{R>0} \frac{1}{R^3} \int_{B(0,R)} |u(x)|^2dx \leq \sup_{R>0}
\frac{1}{R^2} \int_{|x|=R} |u(x)|^2d \sigma_R(x),
\end{equation}
and
\begin{equation}\label{guillermo7}
\begin{array}{ll}
(\frac{1}{2}+\delta)\int_{\mathbb{R}^3}\frac{W(|x|)}{|x|}|u(x)|^2dx=
(\frac{1}{2}+\delta)\int_0^\infty
\frac{W(t)}{t}\int_{|x|=t}|u(x)|^2d\sigma_t(x)dt

\vspace{0.4cm}

\\ \hspace{1cm}  < (\frac{1}{4 }-\delta^2)
\sup_{R>0}\frac{1}{R^2}\int_{|x|=R}|u(x)|^2d\sigma_R(x).
\end{array}
\end{equation}

\subsection{Proof of the theorem \ref{teorema1}}
\label{evolucionrepulsivo}
We start with the proof of the estimate (\ref{juan10}).

For simplicity, for $R>0$ fixed, we introduce the function
$\psi_R$ defined by
\begin{equation}
\label{psiR} \psi_R(x)=\frac{1}{R(1+\frac{|x|^2}{R^2})}.
\end{equation}

Since $\psi_R(x)>\frac{1}{2R}\chi_{B(0,R)}(x),$ to show that
(\ref{juan10}) holds, it is enough to prove that for all $R>0,$
there exists a positive constant $C$ independent of $R,$ such that
\begin{equation*}
\int_{- \infty}^{\infty}\int_{\mathbb{R}^d}
\psi_R(x)|D^{\frac{1}{2}}e^{itH}u_0(x)|^2dxdt \leq C\ \|u_0\|_2^2.
\end{equation*}
This estimate is equivalent to say that the operator
$\psi_R^{\frac{1}{2}}D^{\frac{1}{2}}$ is $H-$smooth (see
\cite{rs}, v.IV or \cite{ky}).
 We will prove, further, that
$\psi_R^{\frac{1}{2}}D^{\frac{1}{2}}$ is $H-$supersmooth, which
means, (see \cite{ky}),  that for all $f\in
 \mathcal{D}(D^{\frac{1}{2}}\psi_R^{\frac{1}{2}}) \subset
L^2(\mathbb{R}^d),$ $\tau\in\mathbb{R}$ and $\epsilon> 0,$ there
exists a  positive constant $C$ independent of $\tau$ and
$\epsilon$ such that
$$
\|\psi_R^{\frac{1}{2}}D^{\frac{1}{2}}R_H(\tau\pm
i\epsilon)D^{\frac{1}{2}}\psi_R^{\frac{1}{2}}f\|_{L^2}\le C\
\|f\|_{L^2}.
$$
In our case, the constant $C$ also has to be
independent of $R.$

This estimate can be obtained, using complex interpolation of
operators, from the following estimate and its dual version:
$$
\|\psi_R^{\frac{1}{2}}D^{1+i\eta}R_H(\tau\pm
i\epsilon)D^{-i\eta}\psi_R^{\frac{1}{2}}f\|_{L^2}\le C\
\|f\|_{L^2}, \qquad\forall\ \eta\in\mathbb{R}.
$$

This inequality can be written as
\begin{equation}
\label{z1peso} \|D^{1+i\eta}R_H(\tau\pm
i\epsilon)D^{-i\eta}f\|_{L^2(\psi_R)}\le C\
\|f\|_{L^2(\psi_R^{-1})}.
\end{equation}

We   see that for any $\gamma\in \mathbb{R},$
\begin{eqnarray}
\label{b1} \|D^{i\gamma}f\|_{L^2(\psi_R)}&\le& C\
\|f\|_{L^2(\psi_R)},
\\
\label{b2} \|D^{-i\gamma}f\|_{L^2(\psi_R^{-1})}&\le& C\
\|f\|_{L^2(\psi_R^{-1})},
\\
\label{rj} \|Df\|_{L^2(\psi_R)}&\le& C\ \|\nabla
f\|_{L^2(\psi_R)},
\end{eqnarray}
with $C$ a positive constant independent of $R.$ Therefore, to
show that (\ref{z1peso}) holds, it will be  enough to prove that
\begin{equation}
\label{nabla} \|\nabla R_H(\tau\pm i\epsilon)f\|_{L^2(\psi_R)}\le
C\ \|f\|_{L^2(\psi_R^{-1})}.
\end{equation}

Estimates (\ref{b1}), (\ref{b2}) and (\ref{rj}) are consequence of
the fact that $\psi_R$ and $\psi_R^{-1}$ are weights in the class
$A_2$ (see \cite{D1}). This means that for any cube $Q$ in
$\mathbb{R}^d,$ there exists a constant $C$ independent of $Q$
such that
$$
\left(\frac{1}{|Q|}\int_Q\psi_R(x)dx\right)
\left(\frac{1}{|Q|}\int_Q\frac{1}{\psi_R(x)}dx\right) \le C.
$$
One can check that the constant $C$ above  is
independent of $R.$

Now, since $\psi_R (x)\le 1/R$ and $\psi_R (x)\le R/|x|^2,$ we
have that
\begin{eqnarray}
\nonumber \| \nabla R_H(\tau\pm i\epsilon)f\|^2_{L^2(\psi_R)}
&\le& \frac{1}{R} \int_{B(0,R)}| \nabla R_H(\tau \pm i \epsilon)
f(x) |^2 dx
\\
\label{izquierda} &&+ R\int_{|x| \geq R} \frac{| \nabla R_H(\tau
\pm i \epsilon) f(x)|^2}{|x|^2} dx.
\end{eqnarray}

We control the second term in the previous inequality by taking
$k\in\mathbb{N}$ such that
\begin{equation}\label{agosto}
2^k<R \leq 2^{k+1}.
\end{equation}
Thus
$$
R\int_{|x| \geq R} \frac{| \nabla R_H(\tau \pm i \epsilon)
f(x)|^2}{|x|^2} dx \leq R \sum_{j\ge k}^{ \infty}
\int_{C_j}\frac{| \nabla R_H(\tau \pm i \epsilon) f(x)|^2}{|x|^2}
dx
$$
$$\leq R\ \| \nabla R_H(\tau \pm i \epsilon) f\|^2_{X} \sum_{j\ge k}^{
\infty} \frac{1}{2^j} \leq C\ \| \nabla R_H(\tau \pm i \epsilon)
f\|^2_{X}.$$
Inserting this in (\ref{izquierda}) we have
that
\begin{equation}
\label{r} \| \nabla R_H(\tau\pm i\epsilon)f\|^2_{L^2(\psi_R)} \le
C\ \| \nabla R_H(\tau \pm i \epsilon) f\|^2_{X}.
\end{equation}

Furthermore, taking into account that $\psi^{-1}_R (x)\ge R,$
$\psi^{-1}_R (x)\ge |x|^2/R,$ and  for $k \in \mathbb{N}$ satifying
(\ref{agosto}), we have that
\begin{equation}\label{faltaba}
\| f\|^2_{X^{*}} \leq R \int_{B(0,R) } |f(x)|^2dx + \frac{1}{R}
\int_{|x| \ge R} |f(x)|^2 |x|^2dx \leq C \|
f\|^2_{L^2(\psi_R^{-1})}.
\end{equation}
After  (\ref{r}), (\ref{agosto1}) and  the above inequality  we get
(\ref{nabla}), from which   (\ref{juan10}) follows.

Arguing in a similar way to show that (\ref{juan14}) holds, we
will prove that the operator of multiplication by the
function $\psi_R^{\frac{1}{2}} \ V^{\frac{1}{4}}$ is
$H$-supersmooth. We   prove then  that for all $f\in
\mathcal{D}(\psi_R^{\frac{1}{2}} \ V^{\frac{1}{4}}) \subset
L^2(\mathbb{R}^d),$ $\tau\in\mathbb{R}$ and $\epsilon> 0,$ there
exists a positive constant $C$ independent of $\tau$,
$\epsilon$ and $R$ such that
$$
\|\psi_R^{\frac{1}{2}} \ V^{\frac{1}{4}} R_H(\tau\pm i\epsilon)
\psi_R^{\frac{1}{2}} \ V^{\frac{1}{4}} f\|_{L^2}\le C\
\|f\|_{L^2}.
$$

Again this estimate can be obtained, using complex interpolation of
operators, from the following estimate and its dual:
$$
\|\psi_R^{\frac{1}{2}}V^{\frac{1+i\gamma}{2}}R_H(\tau\pm
i\epsilon)V^{-i\frac{\gamma}{2}}\psi_R^{\frac{1}{2}}f\|_{L^2} \le
C\ \|f\|_{L^2}, \qquad\forall\ \gamma\in\mathbb{R}.
$$

This inequality can be written as
\begin{equation}
\label{z1pesob} \|V^{\frac{1+i\gamma}{2}}R_H(\tau\pm
i\epsilon)V^{-i\frac{\gamma}{2}}f\|_{L^2(\psi_R)}\le C\
\|f\|_{L^2(\psi_R^{-1})}.
\end{equation}

Since
\begin{eqnarray*}
\|V^{i\gamma}f\|_{L^2(\psi_R)}&\le& C\|f\|_{L^2(\psi_R)},
\\
\|V^{-i\gamma}f\|_{L^2(\psi_R^{-1})}&\le&
C\|f\|_{L^2(\psi_R^{-1})},
\end{eqnarray*}
with $C$ a positive constant independent of $R,$   to
show  (\ref{z1pesob})  we can reduce  to prove that
$$
\|V^{\frac{1}{2}} R_H(\tau\pm i\epsilon)f\|_{L^2(\psi_R)}\le C\
\|f\|_{L^2(\psi_R^{-1})}.
$$

Finally, since $\psi_R (x)\le 1/|x|,$ from (\ref{agosto2}) and
(\ref{faltaba}),we have that
\begin{equation*}
\|V^{\frac{1}{2}} R_H(\tau\pm i\epsilon)f\|_{L^2(\psi_R)} \leq C\
\|R_H(\tau\pm i\epsilon)f\|_{L^2\left(\frac{V(x)}{|x|}\right)}
\end{equation*}
$$ \le C\ \|f\|_{X^*}
\le  C\ \|f\|_{L^2(\psi_R^{-1})}.$$

\hfill $\square$

\begin{remark}\label{ultimoremark}
Estimates (\ref{maricruz0}) and (\ref{tau<epsilon}) hold  for $\gamma=0$,   nevertheless our proof  of  estimate (\ref{principal2repulsivo}) does not  work in  this  case. To  obtain  (\ref{basic})  we need also (\ref{principal2repulsivo}) .  This  is the reason  for  which  theorem \ref{teorema1} is  not  a  generalization  of  the  work \cite{BRV}.
\end{remark}
\section{Some attractive perturbations}
\subsection{Estimates for the resolvent}
In this section, we give a priori estimate for the resolvent of
the operator $H$ defined in the Theorem \ref{teorema2}, that
will be  the  key  to prove  this theorem.
\begin{theorem}
\label{tatractivo} Let $V$ and $n$ be two real valued functions
satisfying the conditions of Theorem \ref{teorema2}, and let u
be  a  solution  of  the  equation
\begin{equation}
\label{ecuacionatractiva} -\Delta u+(V(x)+n(x))\ u\pm i\epsilon u
- \tau u=f, \quad \epsilon\neq 0,\ x\in\mathbb{R}^d \;\;\;\; d\ge
3.
\end{equation}
Then, given $\rho>0$ and $\tau_0>0,$    the following
  estimates  hold:
\begin{equation}\label{basicatractivopv}
\begin{array}{lll}
\hspace{1.5cm}  \|\nabla u\|^2_{X_{\rho}} +
 \sup_{R\ge\rho}\frac{1}{R^2}\int_{|x|=R}|u(x)|^2d\sigma_R(x)

\vspace{0.4cm}

\\ +  \|\ |n|^{\frac{1}{2}}u\|^2_{X_{\rho}} +
(d-3)\int_{\mathbb{R}^d}\frac{|u(x)|^2}{|x|^3}\ dx + \ \tau\
\|u\|^2_{X_{\rho}} + \int_{\mathbb{R}^d}\frac{V(x)}{|x|}|u(x)|^2\
dx

\vspace{0.4cm}

\\ + \chi(d)\int_{\mathbb{R}^3}\frac{W(|x|)}{|x|}|u(x)|^2dx  \leq C(\beta_{\rho})
\left(\|f\|_{X_{\rho}^*}^2+\left\|\frac{f}{|n|^{\frac{1}{2}}}\right\|_{X_{\rho}^*}^2\right),
\hspace{1cm} \forall\ \tau\ge 0 . \end{array}
\end{equation}

\begin{equation}\label{basicatractivo}
\begin{array}{lll}
\hspace{1.5cm} \|\nabla u\|^2_{X_{\rho}} +
\sup_{R\ge\rho}\frac{1}{R^2}\int_{|x|=R}|u(x)|^2d\sigma_R(x)

\vspace{0.4cm}

\\ + \ \|\ |n|^{\frac{1}{2}}u\|^2_{X_{\rho}} +
(d-3)\int_{\mathbb{R}^d}\frac{|u(x)|^2}{|x|^3}\ dx + \ \tau\
\|u\|^2_{X_{\rho}} + \int_{\mathbb{R}^d}\frac{V(x)}{|x|}|u(x)|^2\
dx

\vspace{0.4cm}

\\ \hspace{0.5cm} + \chi(d)\int_{\mathbb{R}^3}\frac{W(|x|)}{|x|}|u(x)|^2dx
 \leq C(\beta_{\rho},\tau_0)\ \|f\|_{X_{\rho}^*}^2, \hspace{2cm}
\forall\ \tau\ge \tau_0. \end{array}
\end{equation}
 where $\chi(d)$ is the function defined in (\ref{tonto7}), $W$ the function
 in (\ref{crepulsivo1}), $C(\beta_{\rho})$ and
$C(\beta_{\rho},\tau_0)$ are constants independent of $\epsilon$
and $\tau$.
 \end{theorem}

\begin{remark}
Estimate    (\ref{basicatractivo}) says that
\begin{equation}\label{playa1}
\| \nabla R_H(\tau \pm i \epsilon) f\|_{X_\rho} \leq C
\|f\|_{X_{\rho}^*},
\end{equation}
with $C$ independent of $\epsilon \neq 0$ and $\tau \geq 0$ and
\begin{equation}\label{playa2}
\tau^{1/2} \|R_H(\tau\pm i\epsilon)f\|_{X_\rho} \leq  C\
\|f\|_{X_{\rho}^*},
\end{equation}
with $C$ independent of $\epsilon \neq 0$ and $\tau \geq \tau_0$.
\end{remark}

\proof

We follow the scheme of  the proof of Theorem
\ref{trepulsivo}.  Let  $u$   denote a solution of
(\ref{ecuacionatractiva}) when $f\in L^2(\mathbb{R}^d).$

We start with the proof of (\ref{basicatractivopv}). This  will be
a consequence of
\begin{equation}
\label{principalatractivo2}
\left\{ \begin{array}{lll}  \|\nabla u\|^2_{X_{\rho}} +
 \sup_{R\ge\rho}\frac{1}{R^2}\int_{|x|=R}|u(x)|^2d\sigma_R(x)
+ \int_{\mathbb{R}^d}\frac{V(x)}{|x|}|u(x)|^2\ dx \vspace{0.4cm}
\\
 + \|\,|n|^{\frac{1}{2}}u\|^2_{X_{\rho}}
 +(d-3)\int_{\mathbb{R}^d}\frac{|u(x)|^2}{|x|^3}\ dx +
\chi(d)\int_{\mathbb{R}^3}\frac{W(|x|)}{|x|}|u(x)|^2dx
\vspace{0.4cm}
\\
+ \tau\ \|u\|^2_{X_{\rho}} \leq
C\left(\|f\|_{X_{\rho}^*}^2+\left\|\frac{f}{|n|^{\frac{1}{2}}}\right\|_{X_{\rho}^*}^2\right)
+ \delta |\epsilon|\|u\|^2_{X_{\rho}}.
\end{array} \right.
\end{equation}
and (\ref{tau<epsilon}), that is true for the solutions of
(\ref{ecuacionatractiva}).

We begin considering the case $d=3$. By (\ref{crepulsivo1}) and
Lemma \ref{lema6} in Appendix 2, we can choose three positive
constants $\alpha$, $\epsilon$ and $\delta$ an a radial function
$\Phi_R$ such that
\begin{itemize}
\item $ \alpha + \frac{\epsilon}{6}+ \int_0^\infty t \left( W(t)+
\frac{\delta}{t}\chi_{(0,\rho)}(t) \right)dt < \frac{1}{2}$. \item
 $\Delta \Phi_R(x)=
-\frac{\epsilon}{R^3}\chi_{(0,R)}(|x|)-\frac{W(|x|)}{|x|}
-\frac{\delta}{|x|^2}\chi_{(0,\rho)}(|x|)$. \item $
\inf_{r>0}\{\Phi'(r),\Phi''(r) \} \geq 0$. \item $ \inf_{r \in
(0,R)}\left\{\frac{\Phi'(r)}{r},\Phi''(r) \right\} \geq \frac{C
\epsilon}{R}$. \item $ \alpha < \Phi'(r) < \kappa < \frac{1}{2},
\;\;\;\; r>0$.
\end{itemize}
 Inserting this function in
(\ref{alberto3}) and following the proof of (\ref{maricruz0}),we
can check that
$$\frac {C_1}{ R}\int_{B(0,R)}|\nabla u(x)|^2\ dx
+\frac{C_1}{R^3}\int_{B(0,R)} |u(x)|^2\ dx
 +
C_1\int_{B(0,\rho)}\frac{|u(x)|^2}{|x|^2 }dx $$
$$+C_1
\int_{\mathbb{R}^3} \frac{W(|x|)}{|x|}|u(x)|^2dx+\alpha
\frac{\gamma}{2}\int_{\mathbb{R}^3}\frac {V(x)}{|x|}|u(x)|^2\ dx
 \leq
 C\int_{\mathbb{R}^3} |f(x)||\nabla u(x)|\ dx
$$
$$+C\int_{\mathbb{R}^3}|f(x)|\frac{|u(x)|}{|x|}\ dx +|\epsilon|
\int_{\mathbb{R}^3}|u(x)|\ |\nabla u(x)|\ dx
+\frac{1}{4}\int_{\mathbb{R}^3}\frac {\nabla n(x)\cdot x}{|x|}\
|u(x)|^2dx,
$$ with $C$ and $C_1$ absolute constants.

 In  a similar  way we
 have also  the  analogous for $d > 3$, see specially
 (\ref{maricruz2}),
 \begin{equation*}
\begin{array}{lll}
  \frac {C_1}{ R}\int_{B(0,R)}|\nabla u(x)|^2\ dx
+\frac{C_1}{R^3}\int_{B(0,R)} |u(x)|^2\ dx+
 C_1(d-3)\int_{\mathbb{R}^d}\frac{|u(x)|^2}{|x|^3}dx

\vspace{0.4cm}

\\ +\frac{\gamma}{2}\int_{\mathbb{R}^d}\frac {V(x)}{|x|}|u(x)|^2\ dx
 \leq
 C\int_{\mathbb{R}^d} |f(x)||\nabla u(x)|\ dx
+C\int_{\mathbb{R}^d}|f(x)|\frac{|u(x)|}{|x|}\ dx

\vspace{0.4cm}

\\
\hspace{1.2cm} +|\epsilon| \int_{\mathbb{R}^d}|u(x)|\ |\nabla
u(x)|\ dx +\int_{\mathbb{R}^d}\frac {\nabla n(x)\cdot x}{|x|}\
|u(x)|^2dx.
\end{array}
\end{equation*}
  Both  together  write  as
\begin{equation}\label{controlatractivo1}
\begin{array}{llll}
  \frac {C_1}{ R}\int_{B(0,R)}|\nabla u(x)|^2\ dx
+\frac{C_1}{R^3}\int_{B(0,R)} |u(x)|^2\ dx+
 C_1(d-3)\int_{\mathbb{R}^d}\frac{|u(x)|^2}{|x|^3}dx

\vspace{0.4cm}

\\
 +\left(\alpha \chi(d)+(1-\chi(d))\right)
   \frac{\gamma}{2}\int_{\mathbb{R}^d}\frac {V(x)}{|x|}|u(x)|^2\ dx
  +C_1 \chi(d) \int_{B(0,\rho)}\frac{|u(x)|^2}{|x|^2
}dx

\vspace{0.4cm}

\\ + C_1\chi \int_{\mathbb{R}^3} \frac{W(|x|)}{|x|}|u(x)|^2dx
 \leq
 C\int_{\mathbb{R}^d} |f(x)||\nabla u(x)|\ dx
+C\int_{\mathbb{R}^d}|f(x)|\frac{|u(x)|}{|x|}\ dx

\vspace{0.4cm}

\\
 +|\epsilon| \int_{\mathbb{R}^d}|u(x)|\ |\nabla
u(x)|\ dx + \left(\frac{1}{4}\chi(d)+(1-\chi(d))\right)
\int_{\mathbb{R}^d}\frac {\nabla n(x)\cdot x}{|x|}\ |u(x)|^2dx.
\end{array}
\end{equation}
This inequality allows us to start proving: for any $\delta>0$
there exist two positive constants $A \equiv A(\delta)$ and $B
\equiv B(\delta, \beta)$ such that
\begin{equation}\label{principalatractivo1}
\left\{ \begin{array}{llll}  A \|\nabla u\|_{X_{\rho}}^2 +A
\sup_{R\ge\rho}\frac{1}{R^3}\int_{B(0,R)} |u(x)|^2\ dx+ A\chi(d)
\int_{B(0,\rho)}\frac{|u(x)|^2}{|x|^2}dx

\vspace{0.4cm}

  \\  + A(d-3) \int_{\mathbb{R}^d}
\frac{|u(x)|^2}{|x|^3}dx + \left(\alpha \chi(d)+(1-\chi(d))\right)
\frac{\gamma}{2} \int_{\mathbb{R}^d} \frac{V(x)}{|x|}|u(x)|^2\ dx

\vspace{0.4cm}

 \\ + A\chi(d) \int_{\mathbb{R}^3}\frac{W(|x|)}{|x|}|u(x)|^2dx \le  B
\left(\|f\|_{X_{\rho}^*}^2+\left\|\frac{f}{|n|^{\frac{1}{2}}}\right\|_{X_{\rho}^*}^2\right)

\vspace{0.4cm}

\\+ \delta(|\epsilon| +
\tau)\|u\|^2_{X_{\rho}}  +\left(\frac{1}{4}
\chi(d)+(1-\chi(d))\right) (\delta+\beta_{\rho}) \|\
|n|^{\frac{1}{2}}u\|^2_{X_{\rho}}  ,
\end{array} \right.
\end{equation}
that  we will use in the proof of (\ref{principalatractivo2}).

\underline{Proof of estimate  (\ref{principalatractivo1})}

 We need to show  that the right hand side of (\ref{controlatractivo1})
is finite in order to take supremun in $R> \rho$.

Take $\varphi=1$ in (\ref{alberto1}), from (\ref{sobolev})
and the fact that $V \geq 0$, we get
\begin{equation}\label{tonto1}
\begin{array}{ll}
 \int_{\mathbb{R}^d} | \nabla u (x)|^2 dx \leq
\int_{\mathbb{R}^d} |n(x)||  u (x)|^2 dx+ \tau \int_{\mathbb{R}^d}
|  u (x)|^2 dx+ \int_{\mathbb{R}^d} |  u ||f| dx

\vspace{0.4cm}

\\ \hspace{0.6cm} \leq ( \|n_1\|_\infty + \tau +1)\|u\|^2_{L^2(\mathbb{R}^d)}
+\|f\|^2_{L^2(\mathbb{R}^d)}+c_1\int_{\mathbb{R}^d} | \nabla u
(x)|^2 dx, \end{array}
\end{equation}
with $0<c_1<1$, the constant in (\ref{sobolev}). From this, it follows
that if $f \in L^2(\mathbb{R}^d)$, then $u$, the solution of
(\ref{ecuacionatractiva}), is in $W^{1,2}(\mathbb{R}^d)$.

From (\ref{tonto1}) and (\ref{sobolev}), it is easy to check that
\begin{equation}\label{julio5}
\| |n|^{\frac{1}{2}}u \|^2_{X_\rho} \leq
\frac{C(\|n_1\|_{\infty},c_1)}{\rho}\left( \|u
\|^2_{L^2(\mathbb{R}^d)} + \|\nabla u
\|^2_{L^2(\mathbb{R}^d)}\right) < \infty.
\end{equation}

On the other hand, we have (\ref{julio0})-(\ref{julio1}) and using
(\ref{julio5}) and (\ref{atractivo}), we have
\begin{equation}\label{malo1}
\begin{array}{llll}
 \int_{\mathbb{R}^d}\frac{\nabla n(x)\cdot x}{|x|}\
|u(x)|^2 dx \leq \sum_{j\ge j_0}\int_{C_j}\frac{\nabla n(x)\cdot
x}{|x|\ |n(x)|}\ |n(x)|\ |u(x)|^2 dx

\vspace{0.4cm}

\\ \hspace{1.7cm} + \int_{B(0,\rho)}\frac{\nabla n(x)\cdot x}{|x|\ |n(x)|}\
|n(x)|\ |u(x)|^2dx

\vspace{0.4cm}

\\ \hspace{1cm} \leq \sum_{j\ge j_0} \sup_{x\in C_j}\frac{\nabla n(x)\cdot x}{|x|\
|n(x)|} \int_{C_j}|n(x)|\ |u(x)|^2dx

\vspace{0.4cm}

\\ + \sup_{x\in B(0,\rho)}\frac{\nabla n(x)\cdot x}{|x|\
|n(x)|} \int_{B(0,\rho)}|n(x)|\ |u(x)|^2dx \leq \beta_{\rho}\
\||n|^{\frac{1}{2}}u\|_{X_{\rho}}^2 < \infty. \end{array}
\end{equation}

We can take the sup  when $R \geq \rho$ in
(\ref{controlatractivo1}) to get
\begin{equation}\label{julio6}
\begin{array}{ll} \hspace{1.2cm}
\|\nabla u\|^2_{X_\rho} + \sup_{R \geq \rho}\frac{1}{R^3}
\int_{B(0,R)}|u(x)|^2dx +\int_{\mathbb{R}^d}\frac{V(x)}{|x|}dx

\vspace{0.4cm}

\\  +(d-3)\int_{\mathbb{R}^d}\frac{|u(x)|^2}{|x|^3}dx
     + \chi(d) \int_{B(0,\rho)}\frac{|u(x)|^2}{|x|^2 }dx +\chi \int_{\mathbb{R}^3}
\frac{W(|x|)}{|x|}|u(x)|^2dx < \infty, \end{array}
\end{equation}
and therefore the left hand  side of (\ref{principalatractivo1})
is finite.

Now let us see the a priori estimate (\ref{principalatractivo1}).
We study the terms on the right hand side of
(\ref{controlatractivo1}).
\begin{equation}
\label{firsttermatractivo} C \int_{\mathbb{R}^d} |f(x)||\nabla
u(x)|\ dx \le C \|f\|^2_{X_{\rho}^*}+ \frac{C_1}{2}  \|\nabla
u\|^2_{X_{\rho}}.
\end{equation}

Let $j_0$ be such that $2^{j_0}<\rho\le 2^{j_0+1}.$ Arguing as in
(\ref{alejandra2}), we have that
\begin{equation*}
C \int_{\mathbb{R}^d}|f(x)|\frac{|u(x)|}{|x|} dx \leq C
\int_{B(0,\rho)}|f(x)|\frac{|u(x)|}{|x|}dx +C \int_{|x|
>\rho}|f(x)|\frac{|u(x)|}{|x|}dx
\end{equation*}
$$ \leq C \int_{B(0,R)}|f(x)|\frac{|u(x)|}{|x|} dx + C \|f\|^2_{X^*_\rho}
+ \frac{C_1}{2} \sup_{R \geq \rho}\frac{1}{R^3}
\int_{B(0,\rho)}|u(x)|^2dx,$$ and
$$C \int_{B(0,\rho)}|f(x)|\frac{|u(x)|}{|x|}dx \leq C
\|f\|^2_{X^*_\rho}+ \frac{C_1(d-3)}{2}
\int_{\mathbb{R}^d}\frac{|u(x)|^2}{|x|^3}dx \hspace{1cm} \text{
for } d>3,$$
$$C\int_{B(0,\rho)}|f(x)|\frac{|u(x)|}{|x|}dx  \leq C \|f\|^2_{X^*_\rho}+ \frac{C_1}{2}
\int_{B(0,\rho)}\frac{|u(x)|^2}{|x|^2}dx \hspace{1cm} \text{ for } d=3.$$

Therefore, for $d \geq 3$ we have
\begin{equation}\label{malo2}
\begin{array}{ll}
 C \int_{\mathbb{R}^d}|f(x)|\frac{|u(x)|}{|x|} dx \leq
C \|f\|^2_{X^*_\rho}+\frac{C_1}{2} \chi(d)
\int_{B(0,\rho)}\frac{|u(x)|^2}{|x|^2}dx

\vspace{0.4cm}

\\  + \frac{C_1(d-3)}{2}
\int_{\mathbb{R}^d}\frac{|u(x)|^2}{|x|^3}dx + \frac{C_1}{2}
\sup_{R \geq \rho}\frac{1}{R^3} \int_{B(0,R)}|u(x)|^2dx.
\end{array}
\end{equation}
To control the third term on the right hand side of
(\ref{controlatractivo1}), use the first inequality of
(\ref{tonto1}) to obtain
$$\int_{\mathbb{R}^d}|\nabla u(x)|^2 dx \le (\tau
+\|n_1\|_{L^{\infty}})\int_{\mathbb{R}^d} |u(x)|^2\ dx
+\int_{\mathbb{R}^d} |f(x)| |u(x)| dx
 +c_1\int_{\mathbb{R}^d}|\nabla u(x)|^2 dx,$$and since $c_1<1$,
\begin{equation*}
|\epsilon|\int_{\mathbb{R}^d}|\nabla u(x)|^2\ dx \leq  C (\tau
+\|n_1\|_{L^{\infty}}) |\epsilon|\int_{\mathbb{R}^d} |u(x)|^2 dx +
C |\epsilon|\int_{\mathbb{R}^d} |f(x)| |u(x)|\ dx.
\end{equation*}
Arguing as in the proof of (\ref{julio2}) but replacing
(\ref{fi1real}) by this last inequality, we obtain for $\delta
> 0$
\begin{equation}\label{controlterminoeatractivo}
\begin{array}{ll}
 |\epsilon|\int_{\mathbb{R}^d}|\nabla u(x)||u(x)|\ dx \leq
C (\tau
+\|n_1\|_{L^{\infty}}+|\epsilon|)^{\frac{1}{2}}\int_{\mathbb{R}^d}|f(x)||u(x)|\
dx

\vspace{0.4cm}

\\ \hspace{0.5cm}  \leq C\left(
\left\|\frac{f}{|n|^{\frac{1}{2}}}\right\|^2_{X_{\rho}^*}+
\|f\|^2_{X_{\rho}^*}\right) + \delta
\||n|^{\frac{1}{2}}u\|^2_{X_{\rho}} + \delta (\tau +|\epsilon|)
\|u\|^2_{X_{\rho}}.
\end{array}
\end{equation}
From (\ref{controlatractivo1}), using (\ref{firsttermatractivo})-
 (\ref{controlterminoeatractivo}),
and  taking the sup  in $R \geq \rho$,  we get
\begin{equation*}
\begin{array}{lllll}
 \sup_{R\ge\rho}\frac{C_1}{R^3}\int_{B(0,R)} |u(x)|^2\ dx +
\left(\alpha \chi(d)+(1-\chi(d))\right)
\frac{\gamma}{2}\int_{\mathbb{R}^d}\frac {V(x)}{|x|}|u(x)|^2\ dx

 \vspace{0.4cm}

\\ +
C_1(d-3)\int_{\mathbb{R}^d}\frac{|u(x)|^2}{|x|^3}dx+C_1 \chi(d)
\int_{B(0,\rho)}\frac{|u(x)|^2}{|x|^2}dx+C_1 \chi(d)
\int_{\mathbb{R}^3} \frac{W(|x|)}{|x|}|u(x)|^2dx

\vspace{0.4cm}

\\
+C_1\|\nabla u\|^2_{X_{\rho}}  \leq
C\left(\|f\|^2_{X_{\rho}^*}+\left\|\frac{f}{|n|^{\frac{1}{2}}}\right\|^2_{X_{\rho}^*}\right)
+ \frac{C_1(d-3)}{2} \int_{\mathbb{R}^d}\frac{|u(x)|^2}{|x|^3}dx

\vspace{0.4cm}

\\ +\frac{C_1\chi(d)}{2}\int_{B(0,\rho)}\frac{|u(x)|^2}{|x|^2}dx
 + \frac{C_1}{2} \sup_{R\ge\rho}\frac{1}{R^3}
\int_{B(0,R)}|u(x)|^2\ dx
 + \frac{C_1}{2} \|\nabla u\|^2_{X_{\rho}}

\vspace{0.4cm}

\\
 +\delta(
|\epsilon| + \tau)\|u\|^2_{X_{\rho}} + \left(\frac{1}{4}
\chi(d)+(1-\chi(d))\right) (\delta+\beta_{\rho})
\||n|^{\frac{1}{2}}u\|^2_{X_{\rho}}, \end{array}
\end{equation*}
and (\ref{principalatractivo1}) follows.

\noindent \underline{Proof of  estimate
(\ref{principalatractivo2})}.

Fix $R \geq \rho$, if we argue as in the proof of
(\ref{principal2repulsivo}), then we get an estimate which is
similar to (\ref{acotacionstep2}). More precisely, if $d \geq 3$
we obtain
\begin{equation}\label{controlatractivo2b}
\begin{array}{llllll}
\frac 1{2R}\int_{B(0,R)}|\nabla u(x)|^2\ dx +
\frac{(d-1)}{4R^2}\int_ {|x|=R}|u(x)|^2d\sigma_R(x)

\vspace{0.4cm}

\\ + \frac{\tau}{2R}  \int_{B(0,R)}  |u(x)|^2 dx+ \frac{1}{2R}
\int_{B(0,R)}|n(x)|  |u(x)|^2 dx +
\frac{d-3}{4}\int_{|x|>R}\frac{|u(x)|^2}{|x|^3}\ dx

\vspace{0.4cm}

\\
+ \frac {\gamma}{2R}\int_{ B(0,R)}V(x)|u(x)|^2\ dx+ \frac
{\gamma}{2}\int_{|x|>R}\frac{V(x)}{|x|}|u(x)|^2\ dx

\vspace{0.4cm}

\\
-(1- \eta)\frac{d-3}{4}\int_{\mathbb{R}^d}\frac{|u(x)|^2}{|x|^3}dx
-\frac{\chi(d)}{2}\int_{\mathbb{R}^3}\frac{W(|x|)}{|x|} |u(x)|^2 \
dx

\vspace{0.4cm}

\\ \leq \int_{\mathbb{R}^d}  |f(x)||\nabla u(x)|\ dx +
\frac{(d+1)}{2}\int_{\mathbb{R}^d}  |f(x)|\frac{|u(x)|}{|x|}\ dx

\vspace{0.4cm}

\\
+ |\epsilon| \int_{\mathbb{R}^d} |\nabla u(x)||u(x)|\ dx + \frac
{1}{2}\int_{\mathbb{R}^d}\frac{V(x)}{|x|}|u(x)|^2\ dx +\frac{1}{2}
\int_{\mathbb{R}^d}\frac{\nabla n(x) \cdot x}{|x|}|u(x)|^2 dx .
\end{array}
\end{equation}

From  (\ref{principalatractivo1}) $$\int_{\mathbb{R}^d}
\frac{V(x)}{|x|}\ |u(x)|^2\
dx+\chi(d)\int_{\mathbb{R}^3}\frac{W(|x|)}{|x|} |u(x)|^2 \ dx <
\infty.$$
Then as in the proof of (\ref{principalatractivo1}), we
have
$$\|\nabla u\|^2_{X_{\rho}}
+ \sup_{R\ge\rho}\frac{1}{R^2}\int_{|x|=R}|u(x)|^2d\sigma_R(x) +
\|\ |n|^{\frac{1}{2}}u\|^2_{X_{\rho}} + \|\
V^{\frac{1}{2}}u\|^2_{X_{\rho}}$$ $$  +
\int_{\mathbb{R}^d}\frac{V(x)}{|x|}|u(x)|^2\ dx +
 \tau\ \|u\|^2_{X_{\rho}} +
(d-3)\int_{\mathbb{R}^d}\frac{|u(x)|^2}{|x|^3}\ dx < \infty.$$

In order to get the a priori estimate, we use
(\ref{principalatractivo1}),
(\ref{malo1}),(\ref{firsttermatractivo}), (\ref{malo2}),
(\ref{guillermo5}), (\ref{guillermo7}) and taking the  sup  in $R
\geq \rho$ in (\ref{controlatractivo2b}), we have for an absolute
constant $C_2$
\begin{equation*}
\begin{array}{lllll} C_2 \|\nabla u\|^2_{X_{\rho}} + \frac{C_2}{4}
\sup_{R\ge\rho}\frac{1}{R^2}\int_{|x|=R}|u(x)|^2d\sigma_R(x)
+\frac{1}{2} \|\ |n|^{\frac{1}{2}}u\|^2_{X_{\rho}} +C_2 \|\
V^{\frac{1}{2}}u\|^2_{X_{\rho}}

\vspace{0.4cm}

\\
 +
\frac{3C_2}{4}\sup_{R\ge\rho}\frac{1}{R^3}\int_{B(0,R)}|u(x)|^2dx
+\frac{\gamma}{2} \int_{\mathbb{R}^d}\frac{V(x)}{|x|}|u(x)|^2\ dx
+  \tau C_2 \|u\|^2_{X_{\rho}}

\vspace{0.4cm}

\\
 +
(d-3)C_2\int_{\mathbb{R}^d}\frac{|u(x)|^2}{|x|^3}\
dx+C_2\chi(d)\int_{\mathbb{R}^3}\frac{W(|x|)}{|x|} |u(x)|^2 \ dx

\vspace{0.4cm}

\\
\le
C\left(\|f\|^2_{X_{\rho}^*}+\left\|\frac{f}{|n|^{\frac{1}{2}}}\right\|^2_{X_{\rho}^*}\right)
 +
 \delta(|\epsilon|+ \tau)\|u\|^2_{X_{\rho}} +\frac{C_2}{2}
\sup_{R\ge\rho}\frac{1}{R^3}\int_{B(0,R)}|u(x)|^2dx

\vspace{0.4cm}

\\
 +
\frac{C_2}{2} \|\nabla u\|^2_{X_{\rho}} +\left(\frac{1}{4 \alpha}
\chi(d)+(1-\chi (d) \right) (\delta
+2\beta_{\rho})\||n|^{\frac{1}{2}}u\|^2_{X_{\rho}}.
\end{array}
\end{equation*}

If $d>3$, we take $\delta$ small and since $\beta_{\rho}<1/4$,
the terms
$$\frac{C_2}{2}   \|\nabla u\|^2_{X_{\rho}}, \hspace{0.3cm} \delta \tau
\|u\|^2_{X_{\rho}},
  \hspace{0.3cm}  (\delta
+2\beta_{\rho})\||n|^{\frac{1}{2}}u\|^2_{X_{\rho}} \hspace{0.3cm}
\textrm{and} \hspace{0.3cm} \frac{C_2}{2}
\sup_{R\ge\rho}\frac{1}{R^3} \int_{B(0,R)}|u(x)|^2\ dx$$ can be
absorbed by their analogous on the left hand side of the above
inequality in order to obtain (\ref{principalatractivo2}).

If $d=3$, we will need that $$\beta_\rho < \alpha <
\frac{1}{2}-\int_0^\infty tW(t)dt.$$

\vspace{0.5cm}

The proof of (\ref{basicatractivo}) is similar to the previous
one, but in this case, we  have to replace $
C(\|f\|^2_{X_{\rho}^*}+\||n|^{-1/2}|f|\|^2_{X_{\rho}^*}) $ by $
C(\tau_0)\|f\|^2_{X_{\rho}^*} $ for all $\tau\ge\tau_0.$ In order
to do this, we have to modify the estimate
(\ref{controlterminoeatractivo}) as follows:
\begin{equation*}
\begin{array}{ll}   |\epsilon|\int_{\mathbb{R}^d}|\nabla u(x)||u(x)| dx
\leq C (\tau
+\|n_1\|_{L^{\infty}}+|\epsilon|)^{\frac{1}{2}}\int_{\mathbb{R}^d}|f(x)||u(x)|
dx

\vspace{0.4cm}

\\
\leq
C\left(\frac{\|n_1\|_{L^{\infty}}}{\tau_0}\right)^{\frac{1}{2}}
\|f\|_{X_{\rho}^*}\||n|^{\frac{1}{2}}u\|_{X_{\rho}} +
C(\tau_0)(\tau
+|\epsilon|)^{\frac{1}{2}}\|f\|_{X_{\rho}^*}\|u\|_{X_{\rho}}
\hspace{1cm} \tau\ge \tau_0. \end{array}
\end{equation*}
\hfill$\square$
\subsection{Proof of  Theorem \ref{teorema2}}
We start with the proof of  estimate (\ref{sofia9a}). Estimate (\ref{sofia9b}) is treated in a
similar way.

Following the proof of the Theorem \ref{teorema1}, for $R\ge\rho$
fixed, we consider the function $\psi_R$ defined by (\ref{psiR}).

As in the proof of (\ref{juan10}), it suffices to prove that
the operator $\psi_R^{\frac{1}{2}}D^{\frac{1}{2}}$ is
$H-$su\-per\-smooth on $[\tau_0,\infty)$ ($\psi_R^{\frac{1}{2}}D^{\frac{1}{2}}P_{[\tau_0, \infty)}$ is
$H-$su\-per\-smooth, see \cite{rs}, vº IV pag.163), which means
 that for all $f\in
D(D^{\frac{1}{2}}\psi_R^{\frac{1}{2}}) \subset L^2(\mathbb{R}^d),$
$\tau\ge \tau_0$ and $\epsilon> 0,$ there exists  a positive
constant $C(\beta_\rho, \tau_0)$ independent of $\tau$, $\epsilon$
and $R$ such that
$$
\|\psi_R^{\frac{1}{2}}D^{\frac{1}{2}}R_H(\tau\pm
i\epsilon)D^{\frac{1}{2}}\psi_R^{\frac{1}{2}}f\|_{L^2} \le
C(\beta_{\rho},\tau_0)\ \|f\|_{L^2}.
$$ (Here, we use that the operators
$e^{itH}$ and $P_{\tau_0}$ commute).

 Arguing as in the proof of
Theorem \ref{teorema1}, but using (\ref{playa1}) instead
(\ref{agosto1}), we obtain (\ref{sofia9a}).

(\ref{sofia9b}) follows in a similar way.

\hfill $\square$
\section{Potentials without smallness assumptions}
\subsection{Estimates for the resolvent}
As in the previous  sections  we  start with  the  Helmholtz equation.
\begin{proposition}\label{proposicion1}
Let $V_1$ and $V_2$ be two real valued potentials as in Theorem \ref{teorema3}, and let u  be  a
solution  of  the  equation
\begin{equation}\label{ecuacionsinpeque}
-\Delta u+(V_1+V_2)\ u\pm i\epsilon u + \tau u=f, \quad
\epsilon\neq 0,\ x\in\mathbb{R}^d\ \hspace{1cm} d\ge 2.
\end{equation}
\begin{itemize}
\item If $d \geq 3$, given $\tau_0>0$ there exits  $B(\tau_0)$
such that
\begin{equation}\label{chosi1}
\tau_0B(\tau_0)=o(1) \hspace{1cm} \tau_0 \rightarrow 0
\end{equation}
 and   the
following a priori estimate holds:
\begin{equation}\label{mayo10}
\begin{array}{ll} \hspace{1cm}
\|\nabla u\|^2_{X_{\tau_0}} + \ \tau\ \|u\|^2_{X_{\tau_0}} +
(d-3)\int_{|x|>\tau_0}\frac{|u(x)|^2}{|x|^3}\ dx

\vspace{0.4cm}

\\  +
\sup_{R \ge \tau_0}\frac{1}{R^2}\int_{|x|=R}|u(x)|^2d\sigma_R(x)
 \leq C \|f\|^2_{X^*_{\tau_0}} \hspace{1.7cm} \text{ for } \tau \geq B(\tau_0)
. \end{array}
\end{equation}
 \item If $d =2$, there exists  $\tau_0>0$ and $B(\tau_0)$
 satisfying (\ref{chosi1})
 and $u$
verifies the following a priori estimate:
\begin{equation}\label{mayo10b}
\begin{array}{ll}
\|\nabla u\|^2_{X_{\tau_0}} + \ \tau\ \|u\|^2_{X_{\tau_0}} +
\sup_{R \ge \tau_0}\frac{1}{R^2}\int_{|x|=R}|u(x)|^2d\sigma_R(x)

\vspace{0.4cm}

\\ \hspace{3cm} \leq C
\|f\|^2_{X^*_{\tau_0}} \hspace{1.7cm} \text{ for } \tau \geq B(\tau_0) .
\end{array}
\end{equation}
\end{itemize}
Here, $C$ is an absolute  positive constant independent of
$\epsilon$ , $\tau$ and $\tau_0$.
\end{proposition}
\proof By a density argument we can assume without loss of
generality that $f\in L^2(\mathbb{R}^d).$

We begin with the case $d \geq 3$.  For $R>\tau_0$ fixed, we
consider the functions $\Phi_R$ and $\varphi_R$ defined by
(\ref{PhiR}) and (\ref{varphiR}). If we put $\Phi=\Phi_R$ and
$\varphi=\varphi_R$ in (\ref{alberto14}), arguing as we did to get
(\ref{acotacionstep2}), we obtain
$$\frac 1{2R}\int_{B(0,R)}|\nabla u(x)|^2\ dx
+\frac{\tau}{2R}\int_{B(0,R)}|u(x)|^2dx$$
$$
+\frac{(d-1)(d-3)}{4}\int_{|x|>R}\frac{|u(x)|^2}{|x|^3}\ dx
+\frac{(d-1)}{8R^2}\int_ {|x|=R}|u(x)|^2d\sigma_R(x)
$$
\begin{equation}\label{acotacionstep2sp}
\leq \frac{1}{2} \int_{\mathbb{R}^d} | \nabla V_1(x) \cdot \nabla
\Phi_R(x)| \ |u(x)|^2dx  + \frac{1}{2} \int_{\mathbb{R}^d}
|V_2(x)|\ |\nabla (|u|^2)(x)|dx
\end{equation}
$$ + \frac{1}{2}  \int_{\mathbb{R}^d}
\varphi_R(x)|V_1(x)+V_2(x)| \ |u(x)|^2dx + \frac{1}{2}
\int_{\mathbb{R}^d} |V_2(x)| \ | \Delta \Phi_R(x)| \ |u(x)|^2dx$$
$$ +\int_{\mathbb{R}^d} |f(x)|\ |\nabla u(x)|dx +
\int_{\mathbb{R}^d}\varphi_R(x)\ |f(x)| \ |u(x)|dx $$ $$+
\frac{1}{2}
 \int_{\mathbb{R}^d}|f(x)| \ | \Delta \Phi_R(x)|\ |u(x)|dx
 + |\epsilon| \int_{\mathbb{R}^d}|u(x)|\ |\nabla u(x)|dx.
$$

As in the proof of theorem \ref{trepulsivo} and theorem
\ref{tatractivo} and using (\ref{c1s})-(\ref{c3s}), we can see
that the left hand side of (\ref{acotacionstep2sp}) is independent
of $R \geq \tau_0$ and it is finite. Then,  taking  the $sup$  in $R
\geq \tau_0$ on the right hand side of (\ref{acotacionstep2sp}),
we have
$$\| \nabla u \|^2_{X_{\tau_0}}+  \tau \|u  \|^2_{X_{\tau_0}}$$
$$+(d-3)\int_{|x|>\tau_0}\frac{|u(x)|^2}{|x|^3}\ dx
+\sup_{R \geq \tau_0} \frac{1}{R^2}\int_
{|x|=R}|u(x)|^2d\sigma_R(x) < \infty.$$

To prove the a priori estimate, we study the terms on the left
hand side of (\ref{acotacionstep2sp}). $C$ will denote an absolute
constant independent of $\tau$, $\tau_0$, $\epsilon$ and $a $ (the
constant in (\ref{c1s})-(\ref{c3s})).

It is easy to check
\begin{equation}\label{drv1}
  \int_{\mathbb{R}^d}  | \nabla V_1(x) \cdot
\nabla \Phi_R(x)| \ |u(x)|^2dx 
 \leq  a  \left( 1 + \frac{2\gamma}{1-2^\gamma} \max
\left\{1,\frac{1}{\tau_0}-1\right\}\right)\|u \|^2_{X_{\tau_0}},
\end{equation}

\begin{eqnarray}\label{v2a}
\int_{\mathbb{R}^d} |V_2(x)|\ |\nabla (|u|^2)(x)|dx \leq 2a^2 \|u
\|^2_{X_{\tau_0}}+\frac{1}{8} \| \nabla u  \|^2_{X_{\tau_0}},
\end{eqnarray}
\begin{eqnarray}\label{v2b}
\frac{1}{2} \int_{\mathbb{R}^d} \varphi_R(x) \ |V_1(x)+V_2(x)| \
|u(x)|^2dx  \leq  a \|u \|^2_{X_{\tau_0}},
\end{eqnarray}

\begin{equation}\label{mayo4}
\frac{1}{2}  \int_{\mathbb{R}^d} |V_2(x)| \ | \Delta \Phi_R(x)| \
|u(x)|^2 \leq  \frac{ad}{2}  \left( 1 + \frac{2\gamma}{1-2^\gamma}
\max \left\{1,\frac{1}{\tau_0}-1\right\}\right)\|u
\|^2_{X_{\tau_0}},
\end{equation}

\begin{equation}\label{mayo5}
\begin{array}{lll}
 \int_{\mathbb{R}^d}\varphi_R(x)|f(x)| \ |u(x)|dx +
 \int_{\mathbb{R}^d}|f(x)| \ | \Delta \Phi_R(x)| \ |u(x)|dx

\vspace{0.4cm}

\\ \leq \frac{d}{R} \int_{B(0,R)}|f(x)| |u(x)|dx + (d-1) \int_{|x|
> R}\frac{1}{|x|}|f(x)|\ |u(x)|dx

\vspace{0.4cm}

\\ \hspace{2cm} \leq d^2C(\delta) \|f \|^2_{X^*_{\tau_0}} +  \|u
\|^2_{X_{\tau_0}}, \end{array}
\end{equation}
and  as in the proof
of estimate (\ref{principalatractivo1})
\begin{equation}\label{terminoesp}
\begin{array}{ll} |\epsilon|\int_{\mathbb{R}^d}|\nabla u(x)||u(x)|\ dx \leq
C\|f\|^2_{X_{\tau_0}^*}+ \delta (\tau+1)\|u\|^2_{X_{\tau_0}}

\vspace{0.4cm}

\\ \;\; + \frac{1}{8}\|\nabla u\|_{X_{\tau_0}} + \frac{d-1}{16} \sup_{R
\ge \tau_0}\frac{1}{R^3}\int_{B(0,R)}|u(x)|^2dx. \end{array}
\end{equation}

Let us  define
\begin{equation}\label{tonto2}
B(\tau_0)= 16 a(d^2-1) \left( 1 + \frac{2^\gamma}{1-2^\gamma} \max
\left\{1,\frac{1}{\tau_0}-1\right\}\right)+a(2a+1).
\end{equation}

From (\ref{acotacionstep2sp}), using
(\ref{drv1})-(\ref{terminoesp}) and taking $$\delta < \min
\left\{\frac{B(\tau_0)}{4},\frac{1}{4}\right\},$$ we get
$$
\frac 1{R}\int_{B(0,R)}|\nabla u(x)|^2\ dx
+\frac{\tau}{R}\int_{B(0,R)}|u(x)|^2dx $$
\begin{equation*}
+(d-3)\int_{|x|>R}\frac{|u(x)|^2}{|x|^3}\ dx +\frac{1}{R^2}\int_
{|x|=R}|u(x)|^2d\sigma_R(x)
\end{equation*}
$$ \leq C\|f\|^2_{X_{\tau_0}^*} + \left(\frac{\tau}{2}+\frac{B(\tau_0)}{4}\right) \|u\|^2_{X_{\tau_0}} +
\frac{1}{2}\|\nabla u\|^2_{X_{\tau_0}} + \frac{d-1}{16} \sup_{R\ge
\tau_0}\frac{1}{R^3} \int_{B(0,R)}|u(x)|^2\ dx.$$

If $\tau >B( \tau_0)$, by taking the $sup$  in the above inequality
and using (\ref{guillermo5}), the terms
$$\left(\frac{\tau}{2}+\frac{B(\tau_0)}{4}\right)
 \|u\|^2_{X_{\tau_0}}, \hspace{0.5cm}  \frac{1}{2}\|\nabla
u\|^2_{X_{\tau_0}} \hspace{0.5cm} \textrm{and} \hspace{0.5cm}
\frac{d-1}{16} \sup_{R\ge \tau_0}\frac{1}{R^3}
\int_{B(0,R)}|u(x)|^2\ dx$$ on the right hand side can be passed to  the l.h.s.
 and we get (\ref{mayo10}).

Id $d=2$, we write (\ref{acotacionstep2sp}) as
$$\frac 1{2R}\int_{B(0,R)}|\nabla u(x)|^2\ dx
+\frac{\tau}{2R}\int_{B(0,R)}|u(x)|^2dx +\frac{1}{8R^2}\int_
{|x|=R}|u(x)|^2d\sigma_R(x)
$$
$$\leq \frac{1}{2} \int_{\mathbb{R}^2}  \nabla V_1(x) \cdot \nabla
\Phi_R(x)\ |u(x)|^2dx  + \frac{1}{2} \int_{\mathbb{R}^2} |V_2(x)|\
|\nabla (|u|^2)(x)|dx
$$
\begin{equation}\label{acotacionstep2spb}
+ \frac{1}{2} \left| \int_{\mathbb{R}^2}
\varphi_R(x)(V_1(x)+V_2(x))\ |u(x)|^2dx\right| + \frac{1}{2}
\left| \int_{\mathbb{R}^2} V_2(x) \Delta \Phi_R(x)
|u(x)|^2dx\right|
\end{equation}
$$ +\int_{\mathbb{R}^2} |f(x)|\ |\nabla u(x)|dx + \left|
\int_{\mathbb{R}^2}\varphi_R(x)f(x) \bar{u}(x)dx \right|$$ $$+
\frac{1}{2} \left| \int_{\mathbb{R}^2}f(x) \Delta \Phi_R(x)
\bar{u}(x)dx \right| + |\epsilon| \int_{\mathbb{R}^2}|u(x)|\
|\nabla u(x)|dx + \frac{1}{4}\int_{|x|>R}\frac{|u(x)|^2}{|x|^3}\
dx. $$
Let us  study the last term in the above inequality.

Using (\ref{tonto2}), we define $\tau_0$, and therefore
$B(\tau_0)$, such that
$$\frac{1}{2B(\tau_0)\tau_0^2} <1$$ and let us take $j_0$  satisfying
$2^{j_0}< \tau_0 \leq 2^{j_0+1}$.

We know, for $R>\tau_0$, that
$$\frac{1}{4}\int_{|x|>R}\frac{|u(x)|^2}{|x|^3}dx \leq
\frac{1}{4\tau_0^3}\int_{|x|>\tau_0}|u(x)|^2dx < \infty$$ and
$$\frac{1}{4}\int_{|x|>R}\frac{|u(x)|^2}{|x|^3}dx \leq \frac{1}{4}
\sum_{j \geq j_0}\int_{|x| \sim 2^j}\frac{|u(x)|^2}{|x|^3}dx  \leq
\frac{1}{4 \tau_0^2} \|u \|^2_{X_{\tau_0}} \leq
\frac{B(\tau_0)}{2} \|u \|^2_{X_{\tau_0}}.$$ Then (\ref{mayo10b})
follows as in the case $d \geq 3 $ since the  term
$\frac{1}{4}\int_{|x|>R}\frac{|u(x)|^2}{|x|^3}dx$ can be taken  to
the l.h.s. and  can be  absorbed  by $ \tau \|u \|^2_{X_{\tau_0}}$
if $ \tau \geq B(\tau_0)$.

\noindent

\begin{remark}
Suppose that $\tau_0=1$ in Proposition \ref{proposicion1} and $d
\geq 3$, then
$$B(a) \equiv B(1)= 16 a(d^2-1)  \left( 1 + \frac{2^\gamma}{1-2^\gamma}
\right)+a(2a+1).$$ Let  be $\eta >0$  and assume  that $a$, the
constant in (\ref{c1s})-(\ref{c3s}), satisfies that $B(a) \leq
\eta$. Then, by Proposition \ref{proposicion1}, we have
$$\|\nabla u\|^2_{X_{1}} + \ \tau\ \|u\|^2_{X_{1}} \leq C
\|f\|^2_{X^*_{1}}, \;\;\; \tau \geq \eta.$$

and if we follow the lines of the proof of Theorem \ref{teorema2}, we can
deduce
\begin{equation*}
 \sup_{R\ge1}\frac{1}{R}\int_{B(0,R)}\int_{-
\infty}^{\infty}
|D^{\frac{1}{2}}e^{itH}\mathcal{P}_{\eta}u_0(x)|^2 dtdx \leq C
\|u_0\|_2^2,
\end{equation*}
and
\begin{equation*}
\tau \sup_{R\ge1}\frac{1}{R}\int_{B(0,R)}\int_{- \infty}^{\infty}
|e^{itH}\mathcal{P}_{\eta}u_0(x)|^2 dtdx \leq C\ \|u_0\|_2^2
\hspace{1cm} \tau \geq \eta.
\end{equation*}

\end{remark}

\noindent

\begin{lemma}\label{lema1}
\emph{(Ikebe-Saito)} Let $\tau_0$  and $\tau_1$ two positive real
numbers such that $\tau_0 < \tau_1$ and $\alpha >0$. Let us
consider the sequences
$$\{ \tau_n+i \epsilon_n \}_{n \in \mathbb{N}} \hspace{1cm} \tau_0
< \tau_n < \tau_1 \hspace{0.5cm} \text{and} \hspace{0.5cm}
0<\epsilon_n < 1,$$ $\{ f_n \}_{n \in \mathbb{N}} \in
L^2((1+|x|)^{1+\alpha}$ and $u_n \in L^2((1+|x|)^{-1-\alpha})$ the
solution of
\begin{equation*}
-\Delta u+(V_1+V_2)\ u\pm i\epsilon_n u + \tau_n u=f_n, \;\;\; \
x\in\mathbb{R}^d\  \hspace{1cm} d\ge 2.
\end{equation*}
with $V_1$ and $V_2$ satisfying the conditions of Theorem
\ref{teorema3}.

If $f_n \; \rightarrow \; f$ in $L^2((1+|x|)^{1+\alpha}$ and
$\tau_n+i \epsilon_n \; \rightarrow \; \tau +i \epsilon$, then
there exists $\lim_{n \rightarrow \infty}u_n=u$ in
$L^2((1+|x|)^{-1-\alpha}$ and  such  that $u$ is the unique solution of
\begin{equation*}
-\Delta u+(V_1+V_2)\ u\pm i\epsilon u + \tau u=f, \hspace{0.4cm}
x\in\mathbb{R}^d  \hspace{1cm} d\ge 2.
\end{equation*}
\end{lemma}

\proof

This lemma is a consequence of  Theorem 1.3, Lemma 1.11 and
Theorem 1.4 of \cite{IS}.

\noindent

\begin{proposition}\label{proposicion2}
Let  $\tau_0$, $\tau_1$, $\alpha$,
 $V_1$ and $V_2$ as in lemma
\ref{lema1} and $ f \in L^2((1+|x|)^{1+\alpha})$. Then, the
solution $u \in L^2((1+|x|)^{-1-\alpha})$  of
\begin{equation*}
-\Delta u+(V_1+V_2)\ u\pm i\epsilon + \tau u=f, \hspace{0.4cm}
x\in\mathbb{R}^d \;\;\;\;\; 0< \epsilon <1  \hspace{1cm} d\ge 2,
\end{equation*}
satisfies, for $\tau_0<\tau<\tau_1$, the a priori estimate
\begin{equation}\label{is}
\|u\|_{L^2((1+|x|)^{-1-\alpha})}+\|\nabla
u\|_{L^2((1+|x|)^{-1-\alpha})} \le
C\|f\|^2_{L^2((1+|x|)^{1+\alpha})},
\end{equation}
where $C$ is a constant that only depends on $\tau_0$,  and
$\tau_1$.
\end{proposition}

\begin{proof}

Suppose that (\ref{is}) is false.  Then, on one hand,  there exist sequences
$\{\epsilon_n\} \in (0,1),$ $\{\tau_n\} \in (\tau_0,\tau_1),$
$\{f_n \} \in L^2((1+|x|)^{1+\alpha})$ and $\{u_n \} \in
L^2((1+|x|)^{-1-\alpha}) $ solution of
\begin{equation*}
-\Delta u+(V_1+V_2)\ u\pm i\epsilon_n + \tau_n u=f_n,
\hspace{0.4cm} x\in\mathbb{R}^d \;\;\;\; d\ge 2,
\end{equation*}
such that
\begin{equation}
\label{cero} \lim_{n\rightarrow
0}\|f_n\|_{L^2((1+|x|)^{1+\alpha})}=0,
\end{equation}
and for all $n \in\mathbb{N},$
\begin{equation}
\label{uno} \|u_n\|_{L^2((1+|x|)^{-1-\alpha})}+\|\nabla
u_n\|_{L^2((1+|x|)^{-1-\alpha})}=1.
\end{equation}
We can take  subsequence $\tau_0 < \tau_m  < \tau_1$ and $0 <
\epsilon_m < 1$ such that
$$\tau_m + i \epsilon_m \; \rightarrow \; \tau + i \epsilon ,$$
with $\tau_0 \leq \tau  \leq \tau_1$ and $0 \leq \epsilon \leq 1$.
Since $f_m \; \rightarrow \; 0$ in $L^2((1+|x|)^{1+\alpha})$, from
Ikebe-Saito's lemma $u_m \; \rightarrow \; 0$ in
$L^2((1+|x|)^{-1-\alpha})$.

On the other hand, if we take
$\varphi=\frac{1}{(1+|x|)^{1+\alpha}}$ in (\ref{alberto1}) and we
use (\ref{c1s}) and (\ref{c2s}), we obtain $$
\int_{\mathbb{R}^d}\frac{|\nabla u_m(x)|^2}{(1+|x|)^{1+\alpha}}dx
\leq $$
$$C(1+\tau_m)\int_{\mathbb{R}^d}\frac{|u_m(x)|^2}{(1+|x|)^{1+\alpha}}dx
+C
\|f_m\|_{L^2((1+|x|)^{1+\alpha})}\|u_m\|_{L^2((1+|x|)^{-1-\alpha})},$$
then
$$\int_{\mathbb{R}^d}\frac{|\nabla
u_m(x)|^2}{(1+|x|)^{1+\alpha}}dx  \, \longrightarrow \; 0,$$ but
this is a contradiction with the fact that $u_m \; \rightarrow \;
0$ in $L^2((1+|x|)^{-1-\alpha})$ and (\ref{uno}).
 \end{proof}

\noindent

\begin{theorem}
\label{tsinpeque} Let $V_1$ and $V_2$ as in proposition
\ref{proposicion1} and $ u$ the solution of
(\ref{ecuacionsinpeque}). Given $\tau_0>0$ and $\alpha>0$, the
following a priori estimate holds:
\begin{equation}\label{basicsinpeque}
 \|\nabla u\|^2_{L^2((1+|x|)^{-1-\alpha})} + \ \tau\
\|u\|^2_{L^2((1+|x|)^{-1-\alpha})}  \leq C(\tau_0)\
\|f\|^2_{L^2((1+|x|)^{1+\alpha})},
\end{equation}
where $C(\tau_0)$ is a positive constant independent of $\epsilon$
and $\tau.$
\end{theorem}
\begin{proof} Again, by a density argument we can assume without loss of
generality that $f\in L^2(\mathbb{R}^d).$

For $\tau_0 >0$ we have:
\begin{equation}\label{mayo1}
L^2((1+|x|)^{1+\alpha}) \subset X^*_{\tau_0} \hspace{0.5cm}
\textrm{and} \hspace{0.5cm} \|f\|_{X^*_{\tau_0}} \leq C
\tau_0^{\frac{1}{2}} \|f\|_{L^2((1+|x|)^{1+\alpha})}
\end{equation}
and
\begin{equation}\label{mayo2}
X_{\tau_0} \subset L^2((1+|x|)^{-1-\alpha}) \hspace{0.5cm}
\textrm{and} \hspace{0.5cm} \|u\|_{L^2((-1-|x|)^{1+\alpha})} \leq
C \tau_0^{\frac{1}{2}} \|u\|_{X_{\tau_0}}.
\end{equation}

Let $B(\tau_0)$ be the constant defined by (\ref{tonto2}). If
$B(\tau_0) \leq \tau_0$, (\ref{basicsinpeque}) is a consequence of
(\ref{mayo1}), (\ref{mayo2}) and (\ref{mayo10}) ( or
(\ref{mayo10b})) .

Suppose that $B(\tau_0) > \tau_0$. If $\tau \in [\tau_0,
B(\tau_0)]$, (\ref{basicsinpeque}) is a consequence of proposition
\ref{proposicion2}. If $ \tau \geq B(\tau_0)$,
(\ref{basicsinpeque}) is a consequence of proposition
\ref{proposicion1}.
\end{proof}

\subsection{Proof of the theorem \ref{teorema3}}
We prove  estimate (\ref{sofia12a}) and (\ref{sofia12b}) can be obtained in a
similar way.

We write $\psi(x)=(1+|x|)^{-1-\alpha}.$  To show that
(\ref{sofia12a}) holds it is enough to prove that the operator
$\psi^{\frac{1}{2}}D^{\frac{1}{2}}$ is $H-$su\-per\-smooth in
$[\tau_0,\infty)$, which means by definition that for all $f \in
\mathcal{D}(\psi^{\frac{1}{2}}D^{\frac{1}{2}}) \subset
 L^2(\mathbb{R}^d),$ $\tau\ge \tau_0$ and
$\epsilon> 0,$ there exists a a positive constant $C$ independent
of $\tau$ and $\epsilon$ such that
$$
\|\psi^{\frac{1}{2}}D^{\frac{1}{2}}R_H(\tau\pm
i\epsilon)D^{\frac{1}{2}}\psi^{\frac{1}{2}}f\|_{L^2} \le
C(\tau_0)\ \|f\|_{L^2}.
$$

Following the proof of theorem \ref{teorema1}, but using theorem
\ref{tsinpeque} instead of (\ref{agosto1}), we get
(\ref{sofia12a}).
 \hfill
  $\square$
\section{Appendix 1}
Here we state some identities that have been used throughout
the paper. They follow by using integration by parts.

Let $u$ be a solution of the equation
\begin{equation}
 -\Delta u+V(x)\ u \pm i\epsilon u - \tau u=f,
  \quad \epsilon\neq 0,\ \;\; x\in\mathbb{R}^d\ \;\;\;d\ge 2.
\end{equation}
If $\varphi$ and $\Phi$ are two real valued functions in
$\mathcal{S}(\mathbb{R}^d)$, then the following identities hold:
\begin{equation}\label{alberto1}
\left| \begin{array}{ll} \int_{\mathbb{R}^d} \varphi(x) |\nabla
u(x)|^2 dx -\frac{1}{2}\int_{\mathbb{R}^d} \Delta \varphi(x)
|u(x)|^2 dx +\int_{\mathbb{R}^d} \varphi(x) V(x)|u(x)|^2 dx

\vspace{0.3cm}

\\ \hspace{1.5cm}  - \tau \int_{\mathbb{R}^d} \varphi(x) | u(x)|^2 dx
=\Re \int_{\mathbb{R}^d} \varphi(x) f(x) \bar{u}(x) dx.
\end{array} \right.
\end{equation}
\begin{equation}
\pm \epsilon  \int_{\mathbb{R}^d} \varphi(x) | u(x)|^2 dx +
\Im\int_{\mathbb{R}^d} \nabla \varphi(x) \cdot\nabla
u(x)\bar{u}(x) dx \label{alberto2} =\Im
\int_{\mathbb{R}^d}\varphi(x) f(x) \bar{u}(x) dx.
\end{equation}
\begin{equation}\label{alberto3}
\left| \begin{array}{lll} \hspace{0.5cm} \int_{\mathbb{R}^d}
\nabla \bar{u}(x) \cdot D^2 \Phi(x)\cdot\nabla  u(x) dx

- \frac{1}{4} \int_{\mathbb{R}^d}\Delta^2\Phi(x) | u(x)|^2 dx

\vspace{0.3cm}

\\ \hspace{1cm} - \frac{1}{2}\int_{\mathbb{R}^d}\nabla V(x) \cdot \nabla\Phi (x)
| u(x)|^2 dx =\pm \epsilon \Im \int_{\mathbb{R}^d} \nabla \Phi (x)
\cdot \nabla \bar{u }(x)  u(x)dx

\vspace{0.3cm}

\\
\hspace{1cm} - \Re \int_{\mathbb{R}^d} f(x) \left(\nabla \Phi (x)
\cdot\nabla \bar{u }(x) + \frac{1}{2} \Delta \Phi(x)
\bar{u}(x)\right)dx .
\end{array} \right.
\end{equation}
\begin{equation}\label{alberto4}
\left| \begin{array}{lllll}  \int_{\mathbb{R}^d} \nabla \bar{u}(x)
\cdot D^2 \Phi(x) \cdot\nabla  u(x) dx -\int_{\mathbb{R}^d}
\varphi(x) |\nabla u(x)|^2 dx

\vspace{0.3cm}

\\
+ \frac{1}{4} \int_{\mathbb{R}^d} \Delta(2\varphi- \Delta\Phi(x))
| u(x)|^2 dx - \frac{1}{2}\int_{\mathbb{R}^d}\nabla V(x)
\cdot\nabla \Phi (x) | u(x)|^2 dx

\vspace{0.3cm}

\\
-\int_{\mathbb{R}^d} \varphi(x)V(x)|u(x)|^2 dx + \tau
\int_{\mathbb{R}^d} \varphi(x) | u(x)|^2 dx

\vspace{0.3cm}

\\

=-\Re\int_{\mathbb{R}^d} \varphi(x) f(x)\bar{u}(x) dx \pm \epsilon
\Im\int_{\mathbb{R}^d} \nabla \Phi (x) \cdot \nabla \bar{u }(x)
u(x)dx

\vspace{0.3cm}

\\
- \Re\int_{\mathbb{R}^d} f(x) \left( \nabla \Phi (x) \cdot
\nabla\bar{u }(x) + \frac{1}{2} \Delta \Phi (x) \bar{u}(x) \right)
dx  . \end{array} \right.
\end{equation}
Furthermore, if we write $V(x)=V_1(x)+V_2(x)$, then
\begin{equation}\label{alberto13}
\left| \begin{array}{lllll} \int_{\mathbb{R}^d} \nabla \bar{u}(x)
\cdot D^2 \Phi(x)\cdot\nabla  u(x) dx

\vspace{0.3cm}

\\
- \frac{1}{4} \int_{\mathbb{R}^d}\Delta^2\Phi(x) | u(x)|^2 dx -
\frac{1}{2}\int_{\mathbb{R}^d}\nabla V_1(x) \cdot \nabla\Phi (x) |
u(x)|^2 dx

\vspace{0.3cm}

\\
+ \frac{1}{2} \int_{\mathbb{R}^d}V_2(x) \nabla\Phi
(x)\cdot\nabla(| u(x)|^2) dx +\frac{1}{2}\int_{\mathbb{R}^d}V_2(x)
\Delta\Phi (x) | u(x)|^2 dx

\vspace{0.3cm}

\\
=- \Re \int_{\mathbb{R}^d} f(x) \left(\nabla \Phi (x) \cdot\nabla
\bar{u }(x) + \frac{1}{2} \Delta \Phi(x) \bar{u}(x)\right)dx

\vspace{0.3cm}

\\
\pm \epsilon \Im \int_{\mathbb{R}^d} \nabla \Phi (x) \cdot \nabla
\bar{u }(x)  u(x)dx. \end{array} \right.
\end{equation}
and
\begin{equation}\label{alberto14}
\left| \begin{array}{llllll} \int_{\mathbb{R}^d} \nabla \bar{u}(x)
\cdot D^2 \Phi(x) \cdot\nabla  u(x) dx -\int_{\mathbb{R}^d}
\varphi(x) |\nabla u(x)|^2 dx

\vspace{0.3cm}

\\
+ \frac{1}{4} \int_{\mathbb{R}^d} \Delta(2\varphi- \Delta\Phi(x))
| u(x)|^2 dx - \frac{1}{2}\int_{\mathbb{R}^d}\nabla V_1(x)
\cdot\nabla \Phi (x) | u(x)|^2 dx

\vspace{0.3cm}

\\
 +\frac{1}{2}\int_{\mathbb{R}^d} V_2(x)
\Delta \Phi(x) |u(x)|^2dx +\frac{1}{2}\int_{\mathbb{R}^d} V_2(x)
\nabla \Phi(x) \cdot \nabla (|u(x)|^2)dx

\vspace{0.3cm}

\\
-\int_{\mathbb{R}^d} \varphi(x)V(x)|u(x)|^2 dx + \tau
\int_{\mathbb{R}^d} \varphi(x) | u(x)|^2 dx

\vspace{0.3cm}

\\
=-\Re\int_{\mathbb{R}^d} \varphi(x) f(x)\bar{u}(x) dx \pm \epsilon
\Im\int_{\mathbb{R}^d} \nabla \Phi (x) \cdot \nabla \bar{u }(x)
u(x)dx

\vspace{0.3cm}

\\
- \Re\int_{\mathbb{R}^d} f(x) \left( \nabla \Phi (x) \cdot
\nabla\bar{u }(x) + \frac{1}{2} \Delta \Phi (x) \bar{u}(x) \right)
dx, \end{array} \right.
\end{equation}
where $\Delta^2$ denotes the bilaplacian and $ D^2 \Phi$ denotes
the Hessian matrix of $\Phi$ with respect to $x_1,\ldots,x_d.$

\section{Appendix 2}

The following lemma is implicit   in
\cite{BRV}.

\begin{lemma}\label{lema6}
Let $\alpha$, $\epsilon$, $\kappa$ and $R$ four positive constants
and $h(t)$ a no negative function in $(0,\infty)$ such that
\begin{equation}\label{izaskum}
 \alpha+\frac{\epsilon}{6}+\int_0^\infty t h(t)dt <  \kappa < \frac{1}{2}.
\end{equation}
Then, we can find a radial function   $\Phi (x) \equiv \Phi(r)$,
$|x|=r$, solution of
\begin{equation}\label{izaskum4}
\Delta^2 \Phi(x)=
-\frac{\epsilon}{R^3}\chi_{(0,R)}(x)-\frac{h(|x|)}{|x|}
\hspace{1cm} x \in \mathbb{R}^3
\end{equation}
such that
\begin{equation}\label{izaskum1}
\inf_{r>0}\{\Phi'(r),\Phi''(r) \} \geq 0,
\end{equation}
\begin{equation}\label{izaskum2}
\inf_{r \in (0,R)}\left\{\frac{\Phi'(r)}{r},\Phi''(r) \right\}
\geq \frac{C \epsilon}{R},
\end{equation}
\begin{equation}\label{izaskum3}
\alpha < \Phi'(r) < \kappa < \frac{1}{2}, \;\;\;\; r>0,
\end{equation}
where  $C$ is an absolute constant.
\end{lemma}
 \proof
In dimension three the bilaplacian of a radial function $\Phi$ has
the  simple expression $\Phi^{iv}+\frac{4}{r}\Phi^{iii}$.

If we integrate (\ref{izaskum4}) we have
\begin{equation}\label{izaskum5}
\Phi'(r)=\psi'(r)+\varphi'(r),
\end{equation}
where
\begin{equation}\label{izaskum6}
\psi'(r)=\frac{1}{r^2}\int_0^r u^2 \int_u^\infty \frac{1}{s^2}
\int_0^sth(t)dtdsdu + c_1+\frac{c_2}{r}
\end{equation}
\begin{equation}\label{izaskum7}
\varphi'(r)=-\frac{1}{r^2}\int_0^r u^2
\int_0^{u}\frac{1}{s^2}\int_0^s t^2 m(t)dtdsdu+c_3r
\end{equation}
with $c_1$, $c_2$ and $c_3$ constant and
$$m(t)=\frac{\epsilon}{R^3}\chi_{(0,R)}(t).$$
If we take derivative in (\ref{izaskum6}) and then we use Fubini's theorem we have
\begin{equation*}
\psi''(r)=\frac{1}{3r^3}\int_0^r
t^3h(t)dt+\frac{1}{3}\int_r^\infty h(t)dt-\frac{2c_2}{r^3}.
\end{equation*}
A similar manipulation gives us
\begin{equation*}
\varphi''(r)=-\int_0^r \frac{1}{s^4}\int_0^st^4m(t)dtds+c_3.
\end{equation*}
In order to have $\psi'' \geq 0$ and $\varphi'' \geq 0$, we take
in the above expressions  $c_2=0$ and $$c_3=\int_0^\infty
\frac{1}{s^4}\int_0^st^4m(t)dtds.$$ Then
\begin{equation}\label{izaskum8}
\psi''(r)=\frac{1}{3r^3}\int_0^r
t^3h(t)dt+\frac{1}{3}\int_r^\infty h(t)dt,
\end{equation}
and
\begin{equation}\label{izaskum9}
\varphi''(r)=\int_r^\infty \frac{1}{s^4}\int_0^st^4m(t)dtds+c_3.
\end{equation}
It is easy to check that $\varphi'(0)=0$ and therefore
\begin{equation}\label{izaskum10}
\varphi'(r)=\int_0^r \varphi''(u)du.
\end{equation}
If we take in (\ref{izaskum6}) $c_1=\alpha$ and we use
(\ref{izaskum5})-(\ref{izaskum10}), we can check
(\ref{izaskum1}), (\ref{izaskum2}) and $\Phi'(r) \geq \alpha$.

To see $\Phi'(r) \leq \kappa $, since $\psi'' \geq 0$, we have
$$\Phi'(r) \leq \alpha + \frac{1}{2} \int_0^\infty th(t)dt +
\varphi'(r).$$ From (\ref{izaskum9}) and (\ref{izaskum10}) we have
that $$\varphi'(r) \leq \frac{\epsilon}{6}$$ and $\Phi'(r) <
\kappa < \frac{1}{2}$ follows by (\ref{izaskum}).

\bibliographystyle{amsplain}

\begin{thebibliography}{MMM}

\bibitem{A}
M. Arai,  \textit{ Absolute continuity of Hamiltonian operators
with repulsive potentials}, Publ. Res. Inst. Math. Sci. 7
(1971-1972) 621-635.

\bibitem{AH} S. Agmon, L. H\"ormander,  \textit{ Asymptotic properties of solutions of differential equations with simple characteristics}, Journal d'Analyse MathŽmatique, 30 (1976), 1-38.

\bibitem{BRV}
J. A. Barcel\'o, A. Ruiz and L. Vega, \textit{Some dispersive
estimates  for Schr\"odinger equations with repulsive potentials},
J. Func. Anal. 236 (2006) 1-24.

\bibitem{CS}
P. Constantin, J. C. Saut, \textit{Local smoothing properties of
dispersive equations}, J. Amer. Math. Soc. 1 (1988) 413-419.




\bibitem{D1} J. Duoandikoetxea,
\textit{Fourier Analysis}, Graduate Sudies in Math. V 29. AMS
(2001).

\bibitem{GMP} I.Gasser; P. A.Markowich; B. Perthame, Dispersion and moment lemmas revisited.  J. Differential Equations  156  (1999),  no. 2, 254--281.
\bibitem{GVV}
M. Golberg, L. Vega and N. Visciglia, \textit{Contraexamples of
Strichartz inequalities for Schr\"{o}dinger equations with
repulsive potential}, Int. Math. Res. Not.(2006) .

\bibitem{IS}
T. Ikebe, Y. Saito \textit{Limiting absorption method and absolute
continuity for the Schr\"odinger operator}, J. Math. Kyoto Univ.
12,3 (1972) 513-542.

\bibitem{ISch}
A. Ionescu and W. Schlag,
 \textit{Agmon-Kato-Kuroda theorems for a large class of perturbations},
  Duke Math. Journal. 131,3 (2006) 397-440.

\bibitem{ky}
T. Kato, K. Yajima \textit{Some examples of smooth operators and
the ssociated smoothing effect}, Math. Phy. v.1,n.4 (1989)
481-496.

\bibitem{LS}
J. E. Lin and W. A. Straus \textit{Decay and scattering of
solutions of a nonlinear Schr\"odinger equation}, J.
Functional Analysis 30, (1978) 245-263.

\bibitem{M}
C. S. Morawetz,  \textit{Time daeay for the non-linear
Klein-Gordon equation }, Pro. Roy. Soc. A306 (1968), 291-296.

\bibitem{PV1}
B. Perthame and L. Vega,  \textit{ Morrey-Campanato Estimates for
Helmholtz Equations}, J. Func. Anal. 164 (1999) 340-355.

\bibitem{PV2}
B. Perthame and L. Vega,  \textit{ Energy concentraction and
Sommerfeld condition for Helmholtz equation with variable index at
infinity}, To appear in GAFA.

\bibitem{rs}
M. Reed and B. Simon, Methods of Modern Mathemathical Physics. V.
III Scattering Theory and V. IV Analysis of Operators, Academic
Press, 1978.

\bibitem{RS}
I. Rodnianski, W. Schlag,  \textit{Time decay for solutions of
Schr\"{o}dinger equation with rough and time-dependent
potentials}, Invent. Math. 155 (2004) 455-513.

\bibitem{R}
F. Rellich,  \textit{Darstelluung der eigenverte von $\delta u +
\lambda u$ durch ein randintegral },  Math. Z. 46 (1940) 635-646.

\bibitem{RV}
A. Ruiz and L. Vega,  \textit{ On local regularity of
Schr\"{o}dinger equations}, Int. Math. Res. Not. 1  (1993) 13-17.

\bibitem{S}
P. Sj\"{o}jolin,  \textit{Regularity of solutions to the
Schr\"{o}dinger equations}, Duke Math. J. 55  (1987) 699-715.

\bibitem{V}
 L. Vega,  \textit{ Schr\"{o}dinger equation: Pointwise convergence
 to the initial data}, Proc. Amer. Math. Soc. 102 (1988) 874-878.

\bibitem{VV}
 L. Vega and N. Visciglia,  \textit{ Asymptotic lower bounds for
 a class of Schr\"{odinger equations}}, To appear in Comm. in Math. Physics.

\end{thebibliography}

\end{document}